\documentclass[a4paper,12pt]{amsart}
\title[Hall algebras]{Hall algebras in the derived category and 
higher rank DT invariants}
\date{}
\author{Yukinobu Toda}

\usepackage{amscd}
\usepackage{amsmath}
\usepackage{amssymb}
\usepackage{amsthm}
\usepackage{float}
\usepackage[dvips]{graphicx}

\usepackage[all,ps,dvips]{xy}

\usepackage{array}
\usepackage{amscd}
\usepackage[all]{xy}

\DeclareFontFamily{U}{rsfs}{%
\skewchar\font127}
\DeclareFontShape{U}{rsfs}{m}{n}{%
<-6>rsfs5<6-8.5>rsfs7<8.5->rsfs10}{}
\DeclareSymbolFont{rsfs}{U}{rsfs}{m}{n}
\DeclareSymbolFontAlphabet
{\mathrsfs}{rsfs}
\DeclareRobustCommand*\rsfs{%
\@fontswitch\relax\mathrsfs}

\theoremstyle{plain}
\newtheorem{thm}{Theorem}[section]
\newtheorem{prop}[thm]{Proposition}
\newtheorem{lem}[thm]{Lemma}

\newtheorem{defi}[thm]{Definition}
\newtheorem{rmk}[thm]{Remark}
\newtheorem{cor}[thm]{Corollary}

\newtheorem{prop-defi}[thm]{Proposition-Definition}
\newtheorem{thm-defi}[thm]{Theorem-Definition}
\newtheorem{lem-defi}[thm]{Lemma-Definition}

\newtheorem{exam}[thm]{Example}

\newdimen\argwidth
\def\db[#1\db]{
 \setbox0=\hbox{$#1$}\argwidth=\wd0
 \setbox0=\hbox{$\left[\box0\right]$}
  \advance\argwidth by -\wd0
 \left[\kern.3\argwidth\box0 \kern.3\argwidth\right]}

\newcommand{\aA}{\mathcal{A}}
\newcommand{\bB}{\mathcal{B}}
\newcommand{\cC}{\mathcal{C}}
\newcommand{\dD}{\mathcal{D}}
\newcommand{\eE}{\mathcal{E}}
\newcommand{\fF}{\mathcal{F}}

\newcommand{\hH}{\mathcal{H}}

\newcommand{\lL}{\mathcal{L}}
\newcommand{\mM}{\mathcal{M}}

\newcommand{\oO}{\mathcal{O}}

\newcommand{\sS}{\mathcal{S}}
\newcommand{\tT}{\mathcal{T}}
\newcommand{\uU}{\mathcal{U}}

\newcommand{\xX}{\mathcal{X}}
\newcommand{\yY}{\mathcal{Y}}

\newcommand{\Supp}{\mathop{\rm Supp}\nolimits}
\newcommand{\Hom}{\mathop{\rm Hom}\nolimits}

\newcommand{\dR}{\mathbf{R}}

\newcommand{\ch}{\mathop{\rm ch}\nolimits}

\newcommand{\Ext}{\mathop{\rm Ext}\nolimits}
\newcommand{\Spec}{\mathop{\rm Spec}\nolimits}
\newcommand{\rank}{\mathop{\rm rank}\nolimits}
\newcommand{\Coh}{\mathop{\rm Coh}\nolimits}

\newcommand{\cneq}{\mathrel{\raise.095ex\hbox{:}\mkern-4.2mu=}}
\newcommand{\eqcn}{\mathrel{=\mkern-4.5mu\raise.095ex\hbox{:}}}

\newcommand{\ext}{\mathop{\rm ext}\nolimits}

\newcommand{\Aut}{\mathop{\rm Aut}\nolimits}

\newcommand{\DT}{\mathop{\rm DT}\nolimits}
\newcommand{\PT}{\mathop{\rm PT}\nolimits}

\newcommand{\Imm}{\mathop{\rm Im}\nolimits}

\newcommand{\GL}{\mathop{\rm GL}\nolimits}

\begin{document}
\maketitle

\begin{abstract}
We remark that 
the combination of the 
works of~Ben-Bassat-Brav-Bussi-Joyce
and Alper-Hall-Rydh
imply the conjectured 
local description 
of the moduli stacks of semi-Schur objects
in the derived category of coherent sheaves on 
projective Calabi-Yau 3-folds. 
This result was 
assumed in the author's previous papers
to apply wall-crossing formulas of 
DT type invariants in the derived category, 
e.g. DT/PT correspondence, rationality, 
etc. 
We also show that the above result 
is applied to prove the higher rank version 
of DT/PT correspondence and rationality.  
\end{abstract}

\section{Introduction}
\subsection{Moduli 
stacks of semi-Schur objects on Calabi-Yau 3-folds}
The Donaldson-Thomas (DT for short) invariants
were introduced by Thomas~\cite{Thom}
as holomorphic analogue of Casson invariants
on real 3-manifolds. 
They count stable 
sheaves on Calabi-Yau 3-folds, and their rank 
one theory is conjectured (and proved in many cases~\cite{PP})
to be 
related to Gromov-Witten invariants~\cite{MNOP}. 
On the other hand, Joyce-Song~\cite{JS} 
and Kontsevich-Soibelman~\cite{K-S}
introduced generalized DT invariants so that 
they also count strictly semistable 
sheaves, and proved their wall-crossing formula. 
It has been expected that 
the wall-crossing formula is 
also applied for DT type invariants counting 
Bridgeland semistable objects~\cite{Brs1}, 
or weak semistable objects~\cite{Tcurve1}, 
in the derived category of coherent sheaves. 
Although almost all the
technical details are parallel with~\cite{JS}, 
there has been a technical 
issue, 
that is a certain local description of the moduli 
stack of semistable objects in the derived
category. 
That technical issue was assumed in the author's previous 
papers~\cite{Tcurve1}, \cite{Tcurve2}, \cite{Tcurve3}, 
\cite{Tsurvey}, \cite{TodBG}, \cite{Tcurve4}, 
so that the proofs of their main formulas 
on DT invariants 
were not mathematically rigorous
except their Euler characteristic version\footnote{The 
Euler characteristic version means the 
formula for the naive Euler characteristics 
of the moduli spaces of sheaves without the weight 
by the Behrend function.}.  

The first purpose of this 
note is to remark that 
the above technical issue
is now settled,
just 
by combining the  
works of~Ben-Bassat-Brav-Bussi-Joyce~\cite{BBBJ}
and Alper-Hall-Rydh~\cite{AHR}. 
As this fact was not 
explicitly mentioned in the literatures, 
we point it out  
in this article.  
The following statement was 
formulated as a conjecture in~\cite[Conjecture~1.2]{Tcurve2}.
\begin{thm}
\emph{(Ben-Bassat-Brav-Bussi-Joyce~\cite{BBBJ}, Alper-Hall-Rydh~\cite{AHR}, Theorem~\ref{thm:alper})}
\label{thm:critical}
Let $X$ be a smooth projective Calabi-Yau 3-fold 
over $\mathbb{C}$
and $\mM$ be the moduli 
stack of semi-Schur objects 
$E \in D^b \Coh(X)$, i.e. 
they satisfy 
$\Ext^{<0}(E, E)=0$. 
For $[E] \in \mM$, let 
$G$ be a maximal reductive 
subgroup of $\Aut(E)$.
Then there is a $G$-invariant analytic 
open neighborhood $V$ of $0$ in 
$\Ext^1(E, E)$, 
a $G$-invariant holomorphic function $f\colon V\to \mathbb{C}$
with $f(0)=df|_{0}=0$, and a smooth morphism 
of complex analytic stacks
\begin{align*}
\Phi \colon ([\{df=0\}/G], 0) \to (\mM, [E])
\end{align*}
of relative dimension $\dim \Aut(E)- \dim G$. 
\end{thm}
The above result was first proved 
when $E \in \Coh(X)$ by Joyce-Song~\cite{JS}
using gauge theory.   
In general, by the work of 
Pantev-To$\ddot{\textrm{e}}$n-Vaquie-Vezzosi~\cite{PTVV}, 
the stack $\mM$
is the truncation of 
a smooth derived stack 
with a $(-1)$-shifted
symplectic structure. Using this fact, 
Ben-Bassat-Brav-Bussi-Joyce~\cite{BBBJ}
showed that $\mM$ has Zariski locally 
an atlas which is written as a critical 
locus of a certain algebraic function. 
In particular, Theorem~\ref{thm:critical} 
was proved by them  
for a Schur object $[E] \in \mM$, 
i.e. $\Hom(E, E)=\mathbb{C}$. 
For a strictly semi-Schur object $[E]\in \mM$, 
the remaining 
issue has been whether we can take 
the function $f$ to be invariant under
the $G$-action. 
This issue was addressed by Bussi~\cite{Bussi}, 
where she 
showed a
result similar to Theorem~\ref{thm:critical}
under the assumption 
that $\mM$ is Zariski locally 
written as a quotient stack of the form 
$[S/\GL_n(\mathbb{C})]$
for a quasi projective variety $S$. 
Still it is not known whether 
the last assumption holds for $\mM$ or not. 
However the work of 
Alper-Hall-Rydh~\cite{AHR}
implies that the stack $\mM$ locally 
near $[E] \in \mM$ 
admits
a smooth morphism of the form 
$[S/G] \to \mM$
for an affine variety $S$ with relative dimension 
$\dim \Aut(E)-\dim G$. 
We will see that the result of~\cite{AHR}
is enough to conclude Theorem~\ref{thm:critical}. 

Now given Theorem~\ref{thm:critical},
all of the 
Hall algebra arguments in~\cite{JS}, \cite{BrI}
are applied 
for any heart $\aA$
of a bounded 
t-structure 
in the derived category of coherent sheaves on 
a Calabi-Yau 3-fold $X$. 
In particular, 
following~\cite{BrI} we can 
construct a  
Poisson algebra homomorphism 
from the regular elements of 
the 
motivic Hall algebra of $\aA$
to the Poisson torus, which 
we will review in Subsection~\ref{subsec:poisson}.
Such a statement is relevant to apply the 
wall-crossing formula
in the derived category.

\subsection{Removing assumptions in the previous papers}
The result of Theorem~\ref{thm:critical}
was conjectured and assumed in 
the 
author's 
previous papers. 
Now we can remove 
that assumption
from 
the results in~\cite[Theorem~1.2]{Tcurve1}, 
\cite[Theorem~1.3]{Tcurve2}, \cite[Theorem~1.3]{Tcurve3}, 
\cite[Theorem~3.11]{Tsurvey}, \cite[Theorem~1.5]{TodBG},
 \cite[Theorem~1.2]{Tcurve4}, 
or change the statement of 
Euler characteristic version to 
that on the honest DT invariants. 

For example, let us focus on the 
DT/PT correspondence in~\cite{Tcurve1}. 
Let $X$ be a smooth projective 
Calabi-Yau 3-fold and 
take 
\begin{align*}
\beta \in H_2(X, \mathbb{Z}), \ n\in \mathbb{Z}. 
\end{align*}
Associated to the above data, 
we have two kinds of curve counting invariants
\begin{align*}
I_{n, \beta} \in \mathbb{Z}, \ 
P_{n, \beta} \in \mathbb{Z}. 
\end{align*}
The invariant $I_{n, \beta}$ 
is the rank one DT invariant~\cite{MNOP}, 
which 
virtually 
counts subschemes $C \subset X$ with 
$\dim C \le 1$, $[C]=\beta$, $\chi(\oO_C)=n$. 
On the other hand, 
$P_{n, \beta}$ 
is the 
\textit{Pandharipande-Thomas (PT for short) stable pair
invariants}~\cite{PT}, which virtually 
counts pairs $(F, s)$ where 
$F$ is a pure one dimensional sheaf with $[F]=\beta$, 
$\chi(F)=n$
and $s \colon \oO_X \to F$ is a morphism 
which is surjective in dimension one.
The above two invariants are known to be related by
\begin{align}\label{DT/PT}
\sum_{n\in \mathbb{Z}}I_{n, \beta} q^n=
M(-q)^{e(X)} \cdot \sum_{n \in \mathbb{Z}} P_{n, \beta}q^n. 
\end{align} 
Here $M(q)$ is the MacMahon function
\begin{align*}
M(q)=\prod_{k\ge 1}(1-q^k)^{-k}. 
\end{align*}
The formula (\ref{DT/PT}) was conjectured in~\cite{PT}, 
its Euler characteristic version was proved in~\cite{Tcurve1}, \cite{StTh}, 
and finally proved by Bridgeland~\cite{BrH}.
As we remarked in~\cite[Appendix, arXiv version]{Tcurve1}, 
the argument of~\cite{Tcurve1} also proves the formula (\ref{DT/PT})
if we knew Theorem~\ref{thm:critical}.  
So we can now prove the formula (\ref{DT/PT}) 
along with the argument of~\cite{Tcurve1}
without 
any assumption. 

In any case, the formula (\ref{DT/PT}) was 
proved by Bridgeland~\cite{BrH} without using Theorem~\ref{thm:critical}, 
so Theorem~\ref{thm:critical} is not 
essential in 
proving (\ref{DT/PT}). 
On the other hand, it seems that 
Theorem~\ref{thm:critical} is essential 
in proving the higher rank version of the formula (\ref{DT/PT}), 
which we discuss in the next subsection. 
Before this, let us discuss the 
difference of the arguments in~\cite{Tcurve1} and~\cite{BrH}.  
In~\cite{Tcurve1}, we
regarded
a subscheme $C \subset X$ as an ideal sheaf $I_C$, 
and a stable pair as a two term complex $(\oO_X \stackrel{s}{\to} F)$. 
As they are rank one objects in the derived category, 
we used the Hall algebra of the heart of some t-structure 
to show the Euler characteristic version of (\ref{DT/PT}). 
On the other hand, in~\cite{BrH}
a subscheme and a stable pair 
were regarded as coherent systems, 
which are one dimensional sheaves together with 
sections. 
The latter viewpoint 
has advantage in the point that everything 
can be worked out in the Hall algebra of 
one dimensional sheaves. 
So the result of Theorem~\ref{thm:critical} for 
coherent sheaves, which was already shown by~\cite{JS}, 
was enough to prove the formula (\ref{DT/PT}). 

However, the above 
interpretation of a sheaf as a coherent 
system is only possible for
a rank one object. 
A higher rank stable sheaf 
is not necessary regarded 
as data $W \otimes \oO_X \to F$ for some finite dimensional 
vector space $W$
and 
a one dimensional sheaf $F$. 
So it is not obvious how to study the higher rank 
DT invariants using Hall algebras of one dimensional 
sheaves as in~\cite{BrH}.

\subsection{Wall-crossing formula for higher rank objects}
The second purpose of this article 
is to  
study higher rank DT invariants, 
for example giving a higher rank analogue of 
the formula (\ref{DT/PT}).  
Here 
we emphasize that, contrary to the rank one case, 
 Theorem~\ref{thm:critical} 
is essential to give a rigorous proof.  
Let us take an ample divisor $\omega$ on $X$ and 
an element
\begin{align}\label{Chern}
(r, D, -\beta, -n) \in H^0(X) \oplus H^2(X) \oplus H^4(X) \oplus H^6(X)
\end{align}
such that $r\ge 1$
and $(r, D \cdot \omega^2)$ coprime. 
Let
\begin{align}\label{intro:D}
\DT(r, D, -\beta, -n) \in \mathbb{Z}
\end{align}
be the DT invariant which virtually 
counts 
$\omega$-slope stable sheaves $E \in \Coh(X)$
whose Chern character coincides 
with (\ref{Chern}). 

We define the notion of a
\textit{PT stable object} as
an object
\begin{align}\label{higherPT}
I^{\bullet} \in D^b \Coh(X)
\end{align}
such that $\hH^i(I^{\bullet})=0$ for $i \neq 0, 1$, 
$\hH^0(I^{\bullet})$ is a
$\omega$-slope stable 
sheaf and $\hH^1(I^{\bullet})$ is zero dimensional.  
This notion appeared in Jason Lo's work~\cite{JLo}
in describing certain polynomial stable objects~\cite{Bay}. 
A PT stable object is a PT stable pair in the rank one case, but 
it is not necessary written 
as a reasonable pair in a higher rank case
(cf.~Subsection~\ref{subsec:PT}). 
We can define the invariant
\begin{align}\label{intro:P}
\PT(r, D, -\beta, -n) \in \mathbb{Z}
\end{align}
which virtually 
counts objects (\ref{higherPT})
whose Chern character is (\ref{Chern}). 
When $(r, D)=(1, 0)$, the 
invariants (\ref{intro:D}), (\ref{intro:P})
coincide with $I_{n, \beta}$, $P_{n, \beta}$ respectively. 
The following is 
a higher rank analogue of the formula (\ref{DT/PT}):  
\begin{thm}\label{thm:higherDTPT}\emph{(Theorem~\ref{thm:DTPTseries})}
For a fixed $(r, D, \beta)$, 
we have the following formula: 
\begin{align}\notag
&\sum_{6n \in \mathbb{Z}} \DT(r, D, -\beta, -n)q^n \\
& \qquad \label{hDTPT}
=M((-1)^r q)^{r \cdot e(X)} \cdot  
\sum_{6n \in \mathbb{Z}} \PT(r, D, -\beta, -n)q^n. 
\end{align}
\end{thm}
We note that the formula (\ref{DT/PT}) is 
a special case of (\ref{hDTPT}) 
by setting $(r, D)=(1, 0)$.  
Indeed, the proof is essentially the same as in~\cite{Tcurve1}
in the rank one case, where 
we studied the wall-crossing phenomena in 
the category generated by $\oO_X$ and 
one dimensional sheaves shifted by $[-1]$
in the derived category.  
We will also construct a similar 
abelian category and investigate 
a wall-crossing phenomena. 
However, in this article we 
simplified several arguments by 
considering certain nested torsion pairs 
in the abelian category 
rather than studying the weak stability conditions. 
Also as we work with the Hall algebra of a t-structure in 
the derived category (rather than that of one dimensional sheaves
as in~\cite{BrH}), the use of Theorem~\ref{thm:critical} is 
essential for the proof. 

The generating series of PT invariants was also conjectured 
to be a rational function in~\cite{PT}, which 
is required to formulate the PT/GW correspondence. 
The rationality conjecture was proved 
for the Euler characteristic version in~\cite{Tolim2}, 
and finally proved by Bridgeland~\cite{BrH}. 
In the higher rank case, we 
have the following similar rationality statement: 
\begin{thm}\label{intro:rational}\emph{(Corollary~\ref{cor:rat})}
For a fixed $(r, D, \beta)$, 
we can write
\begin{align*}
\sum_{6n \in \mathbb{Z}} \PT(r, D, -\beta, -n)q^n
=F(q) \cdot G(q^{\frac{1}{6}})
\end{align*}
where $F(q)$ is the Laurent expansion of a rational 
function 
of $q$, and 
$G(q^{\frac{1}{6}})$ is a Laurent 
polynomial in $q^{\frac{1}{6}}$ with integer coefficients. 
\end{thm}
Again the proof is essentially same as in~\cite{Tolim2}, 
but the proof is much simplified and 
we use Theorem~\ref{thm:critical} for the rigorous proof. 
Also contrary to the rank one case, the 
rational functions $F(q)$, $G(q^{\frac{1}{6}})$ are not necessary 
invariant under $q \leftrightarrow q^{-1}$, 
$q^{\frac{1}{6}} \leftrightarrow q^{-\frac{1}{6}}$, 
respectively in the higher rank case. 
By combining Theorem~\ref{thm:higherDTPT} and Theorem~\ref{intro:rational}, 
we obtain the formula: 
\begin{align*}
\sum_{6n \in \mathbb{Z}} \DT(r, D, -\beta, -n)q^n
=M((-1)^r q)^{r \cdot e(X)} \cdot F(q) \cdot G(q^{\frac{1}{6}}). 
\end{align*}
The above formula is a new structure result on DT invariants 
which is applied for any positive rank. 

\subsection{Related works}
The result similar to Theorem~\ref{thm:critical} was once 
announced by Behrend-Getzler~\cite{BG}. 
Recently, Jiang~\cite{YJ1} proved the 
Behrend function identities 
given in Theorem~\ref{thm:Beh} 
using 
the cyclic $L_{\infty}$-algebra 
technique and 
the unpublished work
by Behrend-Getzler~\cite{BG}.  

So far, there exist some articles 
in which higher rank analogue of DT theory or PT
theory has been studied~\cite{Trk2}, \cite{Stop}, \cite{Nhig}, 
\cite{CDP}, \cite{Shes}. 
In these articles, all the 
higher rank objects were of the form 
$(W \otimes \oO_X \to F)$, which 
do not cover all of the stable sheaves 
as we already mentioned. 
So our situation is much more general
than the above previous articles. 

It is a natural problem to 
extend the results of 
Theorem~\ref{thm:higherDTPT}
and Theorem~\ref{intro:rational}
to the motivic DT invariants 
introduced by Kontsevich-Soibelman~\cite{K-S}. 
Still there exist some technical issues in 
this extension, 
e.g. the existence of an orientation data, 
but the numbers of issues are decreasing due to 
the recent progress on
the rigorous foundation of motivic 
DT theory (cf.~\cite{BDM}, \cite{LeQu}, \cite{BM}, \cite{YJ2}).

\subsection{Acknowledgement}
This article was written while the author 
was visiting to Massachusetts Institute of Technology
in 2016. 
I am grateful to Davesh Maulik for 
the valuable discussions. 
This work is supported by World Premier 
International Research Center Initiative
(WPI initiative), MEXT, Japan, 
and Grant-in Aid
for Scientific Research grant (No.~26287002)
from the Ministry of Education, Culture,
Sports, Science and Technology, Japan, 
and
JSPS Program for Advancing Strategic International Networks to Accelerate the Circulation of Talented Researchers.

\section{Hall algebras in the derived category}

\subsection{$d$-critical stacks}
We first recall Joyce's notion of $d$-critical stacks
introduced in~\cite{JoyceD}. 
For any algebraic stack $\xX$, 
Joyce constructed a sheaf 
of vector spaces 
$\sS_{\xX}^0$
satisfying the following
property. 
For any scheme $V$, a
smooth morphism 
$V \to \xX$, 
 and a
closed embedding $i \colon V \hookrightarrow U$
for a smooth scheme $U$, 
there is an exact sequence
\begin{align*}
0 \to \sS_{\xX}^0|_{V} \oplus \mathbb{C}_{V}
 \to i^{-1}\oO_U/I^2 \stackrel{d}{\to}
 i^{-1} \Omega_U/I \cdot i^{-1} \Omega_U. 
\end{align*}
Here $I \subset i^{-1} \oO_U$ is the ideal sheaf
of functions vanishing on $V$, 
and $\mathbb{C}_V$ is the constant sheaf on $V$. 
For example if 
there is a regular function 
$f \colon U \to \mathbb{A}^1$ with 
$V=\{df=0\}$
and 
$f|_{V^{\rm{red}}}=0$, 
then
\begin{align}\label{sec:S}
s=f+I^2 \in \Gamma(V, \sS_{\xX}^0|_{V}).
\end{align}
By definition 
a pair $(\xX, s)$
for an algebraic stack $\xX$ and 
$s \in H^0(\sS_{\xX}^0)$ is called 
a \textit{d-critical stack}
if for any scheme $V$ and a 
smooth morphism $V \to \xX$, 
the section $s|_{V} \in \Gamma(V, \sS_{\xX}^0|_{V})$
is written as (\ref{sec:S}) 
for some data $(U, f, i)$. In this case, 
the data
\begin{align*}
(V, U, f, i)
\end{align*}
is called a \textit{d-critical chart}. 
Roughly speaking, a
$d$-critical stack 
is an algebraic stack which locally has an atlas
given by the critical locus of some function $f$, and 
the section $s$ remembers the function $f$.  

\subsection{Luna \'etale slice theorem for algebraic stacks}
It is well-known that 
the stack of coherent sheaves on a
projective scheme is Zariski locally written as a 
quotient stack. 
However, such a result is not known for 
the stack of objects in the derived category of 
coherent sheaves. 
The following result 
by 
Alper-Hall-Rydh~\cite{AHR}, 
simplified in the $k=\mathbb{C}$ case, 
will be useful to settle the above issue. 
\begin{thm}\emph{(\cite[Theorem~1.2]{AHR})}\label{thm:alper}
Let $\xX$ be a quasi-separated algebraic stack, 
locally of finite type over 
$\mathbb{C}$
with affine geometric stabilizers. 
Let $x \in \xX$ be a point 
and $G \subset \Aut(x)$
a reductive 
subgroup scheme. 
Then there exists an affine 
scheme $S$ with a
$G$-action, a point $p \in S$
fixed by $G$, and 
a smooth morphism
\begin{align*}
\Phi \colon ([S/G], p) \to (\xX, x). 
\end{align*} 
\end{thm}
We can say more on the above result, which 
we mention in the following remark: 
\begin{rmk}\label{rmk:alper}
From the 
construction of $S$ in the proof of~\cite[Theorem~1.2]{AHR}, 
we have the 2-Cartesian diagrams:
\begin{align*}
\xymatrix{
[\{p\}/G] \ar[r]  \ar[d] \ar@{}[dr]|\square & 
[T_x^{[1]} \xX/G] \ar[r] \ar[d] \ar@{}[dr]|\square &
[S/G] \ar[d]^{\Phi} \\
[\{x\}/\Aut(x)] \ar[r] & 
[T_x^{[1]} \xX/\Aut(x)] 
\ar[r] & \xX. 
}
\end{align*}
Here 
$T_x^{[1]}\xX \subset T_{x}\xX$
is the first order infinitesimal neighborhood at zero, 
the top horizontal morphisms
are induced by a
$G$-equivariant embedding 
$T_x^{[1]}\xX \hookrightarrow S$
sending $0$ to $p$, 
and the bottom horizontal morphisms
are the 
natural closed immersions. 
In particular, 
the morphism $\Phi$ is of relative 
dimension $\dim \Aut(x)-\dim G$, 
the morphism
$G=\Aut(p) \to \Aut(x)$
between stabilizers induced by $\Phi$
coincides with the 
inclusion $G \subset \Aut(x)$, and 
the induced map on tangent spaces
\begin{align*}
d\Phi|_{p} \colon 
T_{p}[S/G] \to T_{x} \xX
\end{align*}
is an isomorphism. 
\end{rmk}

\subsection{Proof of Theorem~\ref{thm:critical}}
In what follows, $X$ is a smooth projective
Calabi-Yau 3-fold over $\mathbb{C}$, i.e. 
$K_X=0$ and $H^1(X, \oO_X)=0$. 
We denote by 
$\mM$ 
the stack of objects 
\begin{align*}
E \in D^b \Coh(X), \ \Ext^{<0}(E, E)=0. 
\end{align*}
By the result of Lieblich~\cite{LIE}, 
the stack $\mM$ is 
an algebraic stack locally of finite type. 
Using the theory of Joyce's $d$-critical 
stacks and 
Theorem~\ref{thm:alper},  
one can show 
an algebraic version of 
Theorem~\ref{thm:critical}. 
The following Theorem~\ref{thm:alg}
 is a
generalization of~\cite[Theorem~4.3]{Bussi}, 
and obviously implies Theorem~\ref{thm:critical}. 
\begin{thm}\label{thm:alg}
For $[E] \in \mM$, 
let $G \subset \Aut(E)$ be the maximal reductive subgroup. 
Then 
there exists 
a smooth affine scheme $U$
with a $G$-action, 
a $G$-invariant point $p \in U$, 
a $G$-invariant regular function $f \colon U \to \mathbb{A}^1$
with $f(p)=df|_{p}=0$, 
and a smooth morphism
\begin{align}\label{thm:smooth}
\Phi \colon ([\{df=0\}/G], p) \to (\mM, [E])
\end{align}
with relative dimension $\dim \Aut(E)-\dim G$. 
Moreover, let $g \in G$ 
acts on $\Ext^1(E, E)$ by 
$\epsilon \mapsto g \circ 
\epsilon \circ g^{-1}$. 
Then there is 
a $G$-equivariant \'etale morphism 
$u \colon U \to \Ext^1(E, E)$ with 
$u(p)=0$, and
under the natural identification 
$T_{[E]}\mM=\Ext^1(E, E)$ we have 
\begin{align*}
d\Phi|_{p} =du|_{p}
\colon T_{p}U \to T_{[E]} \mM. 
\end{align*}
\end{thm}
\begin{proof}
By~\cite{LIE}, the 
stack $\mM$ is 
locally quasi-separated.  
The geometric stabilizer at $[E] \in \mM$
is $\Aut(E)$, which 
is an affine algebraic group 
as it is an open subscheme of $\Hom(E, E)$. 
Therefore applying Theorem~\ref{thm:alper}
and noting Remark~\ref{rmk:alper}, 
there is an affine scheme $S$ with a
$G$-action, 
a $G$-invariant point $p \in S$
and a smooth morphism
\begin{align}\label{smooth}
\Phi \colon [S/G] \to \mM
\end{align}
which sends $p$ to $[E]$
with relative dimension $\dim \Aut(E)-\dim G$. 

By~\cite[Corollary~3.19]{BBBJ},
the stack $\mM$ extends to the 
$d$-critical stack $(\mM, s)$. 
This fact is based on the result by 
Pantev-T\"oen-Vaquie-Vezzosi~\cite{PTVV}
that the stack $\mM$ is the truncation of 
a smooth derived stack with a $(-1)$-shifted symplectic structure. 
 Since (\ref{smooth}) is a smooth morphism, 
by~\cite[Proposition~2.8]{JoyceD}, 
the pull-back
\begin{align*}
\Phi^{\ast} s \in \Gamma([S/G], \sS_{[S/G]}^0)
\end{align*}
gives the $d$-critical stack 
$([S/G], \Phi^{\ast}s)$. 
But by~\cite[Example~2.55]{JoyceD}, 
a $d$-critical structure 
on the quotient stack $[S/G]$ 
is equivalent to a $G$-invariant 
$d$-critical structure on $S$
defined in~\cite[Definition~2.40]{JoyceD}. 
Since $G$ is reductive, $S$ is affine
and $p\in S$ is fixed by $G$, 
we can apply~\cite[Proposition~2.43, Remark~2.47]{JoyceD}
to conclude the following: 
by shrinking $S$ in
a neighborhood of $p\in S$ if necessary, 
there 
exists a $G$-invariant 
critical chart $(S, U, f, i)$
such that $\dim U=\dim T_{p}S$. 
Here by the definition of
$G$-invariant $d$-critical chart 
in~\cite[Definition~2.40]{JoyceD}, 
$U$ is a smooth scheme with a $G$-action,
$f \colon U \to \mathbb{A}^1$ is a $G$-invariant 
function and 
$i \colon S \hookrightarrow U$
 is a $G$-equivariant 
embedding
such that $S=\{df=0\}$
and  
$f|_{S^{\rm{red}}}=0$. 
Therefore the morphism 
(\ref{smooth})
gives a desired smooth morphism (\ref{thm:smooth}). 

By Remark~\ref{rmk:alper}, 
the morphism
$d\Phi|_{p} \colon 
T_p [S/G] \to T_{[E]}\mM$
is an isomorphism. 
Since 
we have 
\begin{align*}
T_{p}[S/G]=T_p S=T_p U, \ 
T_{[E]}\mM=\Ext^1(E, E)
\end{align*}
and the embedding $i$ is $G$-equivariant, 
the morphism $d\Phi|_{p}$ induces the $G$-equivariant 
isomorphism
$T_p U \stackrel{\cong}{\to} \Ext^1(E, E)$. 
Therefore by shrinking $(S, U)$ if necessary, 
we can construct a desired $G$-equivariant 
\'etale morphism 
$u \colon U \to \Ext^1(E, E)$. 
\end{proof}

\subsection{Behrend function identities}
In~\cite{Beh}, Behrend constructed a
canonical 
constructible function
on any scheme, called the \textit{Behrend function}. 
The Behrend function is naturally 
extended to algebraic stacks~\cite{JS}. 
Using Theorem~\ref{thm:critical}, 
the Behrend function
$\nu_{\mM} \colon \mM \to \mathbb{Z}$ on the 
moduli stack $\mM$ in the previous subsection 
is described as follows: 
\begin{align*}
\nu_{\mM}([E])=(-1)^{\hom(E, E)-\ext^1(E, E)}(1-e(M_f(0)))
\end{align*}
where $M_f(0)$ is the Milnor fiber of 
the function $f \colon V \to \mathbb{C}$ 
at $0 \in V$ 
in 
Theorem~\ref{thm:critical}, 
and $e(-)$ is the topological Euler number. 
The result of Theorem~\ref{thm:critical}
implies the analogue of the 
Behrend function identities 
proved for coherent sheaves in~\cite{JS}. 

We introduce some notation. 
First for a constructible function 
$\nu$ on a scheme $M$, we set
\begin{align*}
\int_{M} \nu \ de \cneq \sum_{m\in \mathbb{Z}} m \cdot e(\nu^{-1}(m)). 
\end{align*}
Next, for $E_1, E_2 \in D^b \Coh(X)$
its Euler pairing is defined by
\begin{align}\label{euler}
\chi(E_1, E_2)\cneq \sum_{i\in \mathbb{Z}}
(-1)^i \ext^i(E_1, E_2). 
\end{align}
Below, we fix 
the heart of a bounded t-structure 
$\aA \subset D^b \Coh(X)$
on $D^b \Coh(X)$. 
\begin{rmk}
Note that any object $E \in \aA$
satisfies $\Ext^{<0}(E, E)=0$ 
by the definition of the t-structure, 
hence determines the point $[E] \in \mM$. 
\end{rmk}
We have the following lemma: 
\begin{lem}\label{lem:euler}
For $E_1, E_2 \in \aA$, we have 
\begin{align*}
\chi(E_1, E_2)=\hom(E_1, E_2)-\ext^1(E_1, E_2)+
\ext^1(E_2, E_1)-\hom(E_2, E_1). 
\end{align*}
\begin{proof}
The lemma follows from 
$\ext^i(E_1, E_2)=0$ for $i<0$
by the definition of the t-structure, and 
the identity
$\ext^i(E_1, E_2)=\ext^{3-i}(E_2, E_1)$
from the Serre duality.  
\end{proof}
\end{lem}
Thirdly 
for $E_1, E_2 \in \aA$
and $\xi \in \mathbb{P}(\Ext^1(E_1, E_2))$, we denote 
by $E_{\xi} \in \aA$ the object given by 
the extension class corresponding to $\xi$: 
\begin{align*}
0 \to E_2 \to E_{\xi} \to E_1 \to 0. 
\end{align*}
Under the above preparation, 
we can state the 
generalization of~\cite[Theorem~5.11]{JS} as follows: 
\begin{thm}\label{thm:Beh}
For any heart of a bounded t-structure $\aA \subset D^b \Coh(X)$
and $E_1, E_2 \in \aA$, 
we have the following identities: 
\begin{align*}
&\nu_{\mM}([E_1 \oplus E_2])
=(-1)^{\chi(E_1, E_2)}
\nu_{\mM}([E_1]) \nu_{\mM}([E_2]) \\
&\int_{\xi \in \mathbb{P}(\Ext^1(E_2, E_1))}
\nu_{\mM}(E_{\xi}) \ de 
-
\int_{\xi \in \mathbb{P}(\Ext^1(E_1, E_2))}
\nu_{\mM}(E_{\xi}) \ de \\
&=(\ext^1(E_2, E_1)-\ext^1(E_1, E_2)) \nu_{\mM}([E_1 \oplus E_2]). 
\end{align*}
\end{thm}
\begin{proof}
The result is proved for $\aA=\Coh(X)$
in~\cite[Theorem~5.11]{JS}. 
For a general $\aA$, the result follows 
from the same argument of~\cite[Theorem~5.11]{JS}, 
using Theorem~\ref{thm:critical}
instead of~\cite[Theorem~5.5]{JS}
and noting Lemma~\ref{lem:euler}. 
\end{proof}
\begin{rmk}
By the proof of 
Theorem~\ref{thm:critical},
the weaker statement of Theorem~\ref{thm:critical}
holds after replacing $G$ by its maximal torus. 
As proved in~\cite[Theorem~4.2]{Bussi}, 
the 
latter weaker version is enough to prove 
Theorem~\ref{thm:Beh}. 
\end{rmk}

\subsection{Hall algebras}
We recall the notion of 
motivic Hall algebras 
following~\cite{BrI}. 
Let $\sS$ be 
an algebraic stack locally of finite type 
with affine geometric stabilizers. 
By definition, $K(\mathrm{St}/\sS)$ is defined 
to be the 
$\mathbb{Q}$-vector space 
generated by 
isomorphism classes of symbols
\begin{align}\label{symbol}
[\xX \stackrel{\rho}{\to} \sS]
\end{align}
where $\xX$ is an algebraic stack 
of finite type with affine geometric 
stabilizers. 
The relations are generated by (cf.~\cite[Definition~3.10]{BrI})
\begin{enumerate}
\item 
For every pair of $\xX_1, \xX_2$, we have
\begin{align*}
[\xX_1 \sqcup \xX_2 \stackrel{\rho_1 \sqcup \rho_2}{\to}
 \sS]=
[\xX_1 \stackrel{\rho_1}{\to} \sS]+[\xX_2 \stackrel{\rho_2}{\to}
 \sS]. 
\end{align*}
\item For every geometric bijection 
$\rho \colon \xX_1 \to \xX_2$ and 
a morphism $\rho' \colon \xX_2 \to \sS$, 
we have
\begin{align*}
[\xX_1 \stackrel{\rho \circ \rho'}{\to} \sS]
=[\xX_2 \stackrel{\rho'}{\to} \sS]. 
\end{align*}
\item For every pair of Zariski 
locally trivial
fibrations 
$h_i \colon \xX_i \to \yY$ 
and every morphism $g \colon \yY \to \sS$, we have
\begin{align*}
[\xX_1 \stackrel{g \circ h_1}{\to} \sS]=
[\xX_2 \stackrel{g \circ h_2}{\to} \sS]. 
\end{align*}
\end{enumerate}
Let $\aA \subset D^b \Coh(X)$ be the 
heart of a bounded t-structure such that the substack
\begin{align}\label{ObjA}
\oO bj(\aA) \subset \mM
\end{align}
consisting of objects in $\aA$
is an open substack of $\mM$. 
Then the motivic Hall algebra of $\aA$ is defined as
\begin{align}\label{Hall}
H(\aA) \cneq K(\mathrm{St}/\oO bj(\aA)). 
\end{align}
Note that $H(\aA)$ is naturally a 
$K(\mathrm{St}/\mathbb{C}) \cneq 
K(\mathrm{St}/\Spec \mathbb{C})$-module 
by 
\begin{align*}
[\xX \to \Spec \mathbb{C}] \cdot 
[\yY \stackrel{\rho}{\to} \oO bj(\aA)]
=[\xX \times \yY \stackrel{p}{\to} \yY \stackrel{\rho}{\to} \oO bj(\aA)]
\end{align*}
where $p$ is the projection. 
There is an associative 
$K(\mathrm{St}/\mathbb{C})$-algebra
structure $\ast$ 
on $H(\aA)$
based on the Ringel-Hall algebras.
Let $\eE x(\aA)$ be the 
stack of short exact sequences 
\begin{align*}
0 \to E_1 \to E_3 \to E_2 \to 0
\end{align*}
in $\aA$
and $p_i \colon \eE x(\aA) \to 
\oO bj(\aA)$ the 
1-morphism sending 
$E_{\bullet}$ to $E_i$. 
The $\ast$-product on $H(\aA)$ 
is given by 
\begin{align*}
[\xX_1 \stackrel{\rho_1}{\to} \oO bj(\aA)]
\ast [\xX_2 \stackrel{\rho_2}{\to} \oO bj(\aA)]
=[\xX_3 \stackrel{\rho_3}{\to} \oO bj(\aA)]
\end{align*}
where $(\xX_3, \rho_3=p_3 \circ (\rho_1', \rho_2'))$ is given by
the following Cartesian diagram
\begin{align*}
\xymatrix{
\xX_3 \ar[r]^{\hspace{-5mm}(\rho_1', \rho_2')}\ar[d] \ar@{}[dr]|\square
& \eE x(\aA) \ar[d]^{(p_1, p_2)}
  \ar[r]^{p_3} &
\oO bj(\aA) \\
\xX_1 \times \xX_2  \ar[r]^{\hspace{-5mm}(\rho_1, \rho_2)} 
& \oO bj(\aA)^{\times 2}.
& }
\end{align*}
The unit is given by 
$1=[\Spec \mathbb{C} \to \oO bj(\aA)]$
which corresponds to $0\in \aA$. 

Let $\Gamma$ be the image of the Chern character map
\begin{align*}
\Gamma \cneq \Imm (\ch \colon K(X) \to H^{\ast}(X, \mathbb{Q})). 
\end{align*}
Then $\Gamma$ is a finitely generated free abelian group. 
The stack (\ref{ObjA}) decomposes into the disjoint union 
of 
open and closed substacks 
\begin{align*}
\oO bj(\aA)=\coprod_{v \in \Gamma} \oO bj_{v}(\aA)
\end{align*}
where $\oO bj_{v}(\aA)$ is the stack of 
objects in $\aA$ with Chern character $v$. 
The algebra $H(\aA)$ is $\Gamma$-graded
\begin{align*}
H(\aA)=\bigoplus_{v \in \Gamma} H_{v}(\aA)
\end{align*}
where $H_{v}(\aA)$ is spanned by 
$[\xX \to \oO bj(\aA)]$
which factors through 
$\oO bj_{v}(\aA) \subset \oO bj(\aA)$. 

\subsection{Poisson algebra homomorphism}\label{subsec:poisson}
It is easy to see that the affine line
\begin{align*}
\mathbb{L} \cneq [\mathbb{A}^1 \to \Spec \mathbb{C}]
\in K(\mathrm{St}/\mathbb{C})
\end{align*}
is an invertible 
element.
We define
the subalgebra
\begin{align*}
K(\mathrm{Var}/\mathbb{C})[\mathbb{L}^{-1}] 
\subset K(\mathrm{St}/\mathbb{C})
\end{align*}
to be generated by 
$\mathbb{L}^{-1}$ and 
$[Y \to \Spec \mathbb{C}]$
for a variety $Y$. 
The $K(\mathrm{Var}/\mathbb{C})[\mathbb{L}^{-1}]$-submodule
\begin{align}\label{reg}
H^{\rm{reg}}(\aA) \subset H(\aA)
\end{align} 
is defined to be spanned by 
$[Z \to \oO bj(\aA)]$ 
so that $Z$ is a variety. 
An element of $H^{\rm{reg}}(\aA)$ is called \textit{regular}. 
By~\cite[Theorem~5.1]{BrI}, 
the submodule (\ref{reg}) 
is indeed a subalgebra
with respect to the $\ast$-product. 
Moreover the quotient
\begin{align*}
H^{\rm{sc}}(\aA)
 \cneq H^{\rm{reg}}(\aA)/(\mathbb{L}-1)H^{\rm{reg}}(\aA)
\end{align*}
is a commutative algebra. 
Therefore for $f, g \in H^{\rm{sc}}(\aA)$, 
we can define the following bracket
on $H^{\rm{sc}}(\aA)$: 
\begin{align*}
\{f, g\} \cneq
\frac{f \ast g-g \ast f}{\mathbb{L}-1}. 
\end{align*}
By the $\ast$-product together with the above 
bracket $\{-, -\}$, 
we have the Poisson algebra structure on $H^{\rm{sc}}(\aA)$. 

We define another Poisson algebra $C(X)$ to be
\begin{align*}
C(X) \cneq \bigoplus_{v \in \Gamma} \mathbb{\mathbb{Q}} \cdot c_{v}. 
\end{align*}
Note that by the Riemann-Roch theorem, 
the Euler pairing (\ref{euler})
descends to the anti-symmetric bi-linear form
\begin{align*}
\chi \colon \Gamma \times \Gamma \to \Gamma. 
\end{align*}
We will only use the following computation, 
which is an easy consequence of the Riemann-Roch 
theorem: 
\begin{align}\label{RR}
\chi((0, 0, -\beta_1, -n_1), (r, D, -\beta_2, -n_2))
=rn_1-D \beta_1. 
\end{align}
Here and in what follows, we use the notation (\ref{Chern}) 
for the elements in $\Gamma$. 
The
 $\ast$-product on $C(X)$ is defined by 
\begin{align*}
c_{v_1} \ast c_{v_2}=(-1)^{\chi(v_1, v_2)}c_{v_1+v_2}. 
\end{align*}
The Poisson bracket on $C(X)$ is defined by
\begin{align*}
\{c_{v_1}, c_{v_2}\} =(-1)^{\chi(v_1, v_2)}\chi(v_1, v_2) c_{v_1+v_2}. 
\end{align*}
The result
of Theorem~\ref{thm:critical} leads to the following result: 
\begin{thm}\emph{(\cite[Theorem~5.2]{BrI})}
There is a Poisson algebra homomorphism 
\begin{align*}
I \colon H^{\rm{sc}}(\aA) \to C(X)
\end{align*}
such that for a variety $Z$ with 
a morphism $\rho \colon Z \to \oO bj(\aA)$
which factors through $\oO bj_v(\aA)$, 
we have
\begin{align*}
I([Z \stackrel{\rho}{\to} \oO bj(\aA)])=
\left(
\int_{Z} \rho^{\ast} \nu_{\mM} \right)
\cdot c_{v}. 
\end{align*}
Here $\nu_{\mM}$ is the Behrend function on $\mM$
restricted to $\oO bj(\aA)$. 
\end{thm}
\begin{proof}
By Theorem~\ref{thm:Beh}, 
the assumption in~\cite[Theorem~5.2]{BrI} is satisfied, 
hence we conclude the result. 
\end{proof}

\section{Higher rank DT/PT correspondence}\label{sec:DTPT}
In this section, we prove Theorem~\ref{thm:higherDTPT}.

\subsection{Donaldson-Thomas invariants}
Below, we fix an ample divisor $\omega$ on 
a Calabi-Yau 3-fold $X$. 
Let us take an element
\begin{align}\label{rdbn}
v=
(r, D, -\beta, -n) \in \Gamma
\end{align}
in the notation (\ref{Chern}). 
Since $\Gamma$ is the image of the Chern character map, 
we may regard
\begin{align*}
r \in \mathbb{Z}, \ D \in H^2(X, \mathbb{Z}), \
2\beta \in H^4(X, \mathbb{Z}), \ 6n \in \mathbb{Z}. 
\end{align*}
We also assume that 
\begin{align}\label{coprime}
r\in \mathbb{Z}_{\ge 1}, \ 
\mathrm{g.c.d.}(r, D \cdot \omega^2)=1. 
\end{align}
Recall that for 
a coherent sheaf $E$ on $X$, 
its slope function 
$\mu_{\omega}$ is defined 
by
\begin{align*}
\mu_{\omega}(E) \cneq \frac{c_1(E) \cdot \omega^2}{\rank(E)} \in \mathbb{Q} \cup \{\infty\}. 
\end{align*}
A sheaf $E$ is called 
$\mu_{\omega}$-\textit{(semi)stable} if for any 
non-trivial subsheaf $F \subset E$, we have
\begin{align*}
\mu_{\omega}(F) <(\le) \mu_{\omega}(E/F). 
\end{align*}
By the coprime condition (\ref{coprime}), the 
moduli space
\begin{align}\label{MDT}
M_{\rm{DT}}(r, D, -\beta, -n)
\end{align}
of $\mu_{\omega}$-semistable sheaves with Chern character 
(\ref{rdbn})
consists of $\mu_{\omega}$-stable sheaves, 
and it is a projective scheme~\cite{Hu}. 
Also by the CY3 condition of $X$, 
the moduli space (\ref{MDT})
is equipped with 
a symmetric perfect obstruction theory
and the associated zero dimensional virtual class~\cite{Thom}, ~\cite{BBr}.
The \textit{Donaldson-Thomas (DT) invariant} is defined 
by
\begin{align}\label{def:DT}
\DT(r, D, -\beta, -n) \cneq 
\int_{[M_{\rm{DT}}(r, D, -\beta, -n)]^{\rm{vir}}}1. 
\end{align}
By fixing $(r, D)$, we define the generating series
\begin{align}\label{DT:series}
\DT_{r, D}(q, t) \cneq 
\sum_{\beta, n} \DT(r, D, -\beta, -n) q^n t^{\beta}. 
\end{align}

\subsection{Higher rank Pandharipande-Thomas theory}\label{subsec:PT}
We introduce the higher rank version of 
\textit{Pandharipande-Thomas (PT) theory}. 
The following notion was 
found in Lo's work~\cite{JLo} in the analysis of 
Bayer's polynomial stability conditions~\cite{Bay}. 
\begin{defi}\label{def:PT}
An object $I^{\bullet} \in D^b \Coh(X)$ is called PT-(semi)stable
if the following conditions are satisfied:
\begin{enumerate}
\item $\hH^i(I^{\bullet})=0$ for $i \neq 0, 1$. 
\item $\hH^0(I^{\bullet})$ is $\mu_{\omega}$-(semi)stable and 
$\hH^1(I^{\bullet})$ is zero dimensional. 
\item $\Hom(Q[-1], I^{\bullet})=0$ for any zero dimensional 
sheaf $Q$. 
\end{enumerate}
\end{defi}
Let us consider its relationship to the 
rank one PT theory in~\cite{PT}. 
\begin{exam}
Let $\eE$ be a locally free $\mu_{\omega}$-(semi)stable 
sheaf on $X$, 
$F$ a pure one dimensional sheaf on $X$
and $s \colon \eE \to F$
a morphism which is surjective in dimension one. 
Then the object
\begin{align}\label{PT:obj}
I^{\bullet}=(\eE \stackrel{s}{\to} F) \in D^b \Coh(X)
\end{align}
with $\eE$ located in degree zero is a PT (semi)stable object. 
If $\eE=\oO_X$, then it is nothing  
but the stable pair in~\cite{PT}. 
\end{exam}
If an object $I^{\bullet} \in D^b \Coh(X)$ is 
rank one, 
then it is proved in~\cite[Lemma~3.11]{Tcurve1}
that $I^{\bullet}$ is PT semistable 
if and only if it is a two term complex (\ref{PT:obj})
with $\eE$ a line bundle. 
The same argument easily shows the following: 
\begin{lem}
A PT-semistable object $I^{\bullet} \in D^b \Coh(X)$ with 
$\rank(I^{\bullet})>0$ is
quasi-isomorphic to a two term complex of the form (\ref{PT:obj})
if and only if $\hH^0(I^{\bullet})^{\vee \vee}$ is locally 
free. 
\end{lem}
If the rank is bigger than one, 
a reflexive sheaf 
may not be locally free. 
The following example gives an 
exotic PT stable object. 
\begin{exam}
Let $\uU$ be a $\mu_{\omega}$-stable reflexive 
sheaf which is not locally free at $x \in X$. 
Then $\Ext^2(\oO_x, U)=\Ext^1(U, \oO_x)^{\vee} \neq 0$.
The object $I^{\bullet}$ given by a non-trivial extension  
\begin{align*}
\uU \to I^{\bullet} \to \oO_x[-1]
\end{align*}
is a PT stable object. 
However it is easy to see that $I^{\bullet}$ is 
not quasi-isomorphic to (\ref{PT:obj})
for any torsion free sheaf $\eE$
and a one dimensional sheaf $F$. 
\end{exam}
Let
\begin{align}\label{MPT}
M_{\rm{PT}}(r, D, -\beta, -n)
\end{align}
be the moduli space of 
PT-semistable objects in $D^b \Coh(X)$
with Chern character (\ref{rdbn}).
By the coprime condition (\ref{coprime}), 
the moduli space (\ref{MPT}) consists of 
only PT-stable objects. 
By~\cite{JLo2}, the 
moduli space (\ref{MPT}) is a proper algebraic 
space of finite type. 
Because of the CY3 condition of $X$, it 
is also equipped with a symmetric perfect obstruction theory~\cite{HT2}, 
and the associate zero dimensional virtual fundamental class. 
Similarly to (\ref{def:DT}), 
 we can define the invariant
\begin{align}\label{def:PT}
\PT(r, D, -\beta, -n) \cneq 
\int_{[M_{\rm{PT}}(r, D, -\beta, -n)]^{\rm{vir}}}1. 
\end{align}
Also similarly to (\ref{DT:series}), 
we consider the generating series
\begin{align}\label{PT:series}
\PT_{r, D}(q, t) \cneq 
\sum_{\beta, n} \PT(r, D, -\beta, -n) q^n t^{\beta}. 
\end{align}
\subsection{Tilting of $\Coh(X)$}
Recall that a \textit{torsion pair} on an
abelian 
category $\aA$ is a pair of 
full subcategories $(\tT, \fF)$ on $\aA$
such that 
\begin{enumerate}
\item We have $\Hom(T, F)=0$ for $T \in \tT$, $F \in \fF$. 
\item For any $E \in \aA$, there is an exact sequence
$0 \to T \to E \to F\to 0$
with $T \in \tT$ and $F \in \fF$. 
\end{enumerate}
The category $\tT$ is called the \textit{torsion part} 
of the torsion pair $(\tT, \fF)$. 
For a torsion pair $(\tT, \fF)$ on $\aA$, 
its \textit{tilting}
is defined by
\begin{align*}
\aA^{\dag} \cneq \langle \fF, \tT[-1] \rangle
\subset D^b(\aA). 
\end{align*}
Here $\langle \ast \rangle$ means the extension closure. 
The tilting $\aA^{\dag}$ is known to be the 
heart of a bounded t-structure on $D^b(\aA)$
(cf.~\cite{HRS}). 
More generally, we introduce the following notion:
\begin{defi}\label{def:nest}
Let $\aA$ be an 
exact category and 
$\fF_1, \fF_2, \cdots, \fF_n$
full subcategories of $\aA$. 
Then we write
\begin{align*}
\aA=\langle \fF_1, \fF_2, \cdots, \fF_n \rangle
\end{align*}
if the following
conditions are satisfied:
\begin{enumerate}
\item We have $\Hom(F_i, F_j)=0$ for
$F_i \in \fF_i$ and $F_j \in \fF_j$ with $i<j$. 
\item For any $E \in \aA$, there is a filtration 
\begin{align}\label{filt}
0=E_0 \subset E_1 \subset \cdots \subset E_n=E
\end{align}
in $\aA$
such that $E_i/E_{i-1} \in \fF_i$. 
\end{enumerate}
\end{defi}
\begin{rmk}\label{rmk:filt}
In the above definition, the 
$n=2$ case corresponds to the torsion pair. 
Similarly to the torsion pair, 
the filtration (\ref{filt})
is unique if it exists. 
The filtration (\ref{filt}) is an analogue of the 
Harder-Narasimhan filtration in 
some stability condition. 
Also for any $1\le i\le n$, 
we have the torsion pair in $\aA$
\begin{align*}
\aA=\left\langle \langle \fF_1, \cdots, \fF_i \rangle, 
\langle \fF_{i+1}, \cdots, \fF_{n} \rangle \right\rangle.
\end{align*} 
\end{rmk}

Let $(X, \omega)$ be as in the previous subsections. 
For an interval $I \subset \mathbb{R} \cup \{\infty\}$, we set
\begin{align*}
\Coh_{I}(X) &\cneq \left\langle E \in \Coh(X): \begin{array}{c}
E \mbox{ is } \mu_{\omega} \mbox{-semistable } \\
\mbox{ with }
\mu_{\omega}(E) \in I \end{array}
\right\rangle \cup \{0\}. 
\end{align*}
Let us take the numerical class as in 
(\ref{rdbn}). 
We set
\begin{align}\label{def:mu}
\mu \cneq \frac{D \cdot \omega^2}{r} \in \mathbb{Q}. 
\end{align}
By the existence of Harder-Narasimhan filtrations 
with respect to the $\mu_{\omega}$-stability, we 
have the following torsion pair in $\Coh(X)$:
\begin{align*}
\Coh(X)=\langle 
\Coh_{>\mu}(X), \Coh_{\le \mu}(X) \rangle. 
\end{align*}
We take its tilting 
\begin{align}\label{mu:tilt}
\aA_{\mu} \cneq \langle 
\Coh_{\le \mu}(X), \Coh_{>\mu}(X)[-1] \rangle. 
\end{align}
By the construction, for any $E \in \aA_{\mu}$ we have
\begin{align*}
\rank(E) \cdot D\omega^2 -c_1(E) \omega^2 \cdot r \ge 0. 
\end{align*}
Therefore the category
\begin{align*}
\bB_{\mu} \cneq \{ E \in \aA_{\mu} : 
\rank(E) \cdot D\omega^2 -c_1(E) \omega^2 \cdot r =0
\}
\end{align*}
is an abelian subcategory of $\aA_{\mu}$. 
From the construction of (\ref{mu:tilt}), it is easy to see that
\begin{align}\label{Bu}
\bB_{\mu}=\langle \Coh_{\mu}(X), \Coh_{\le 1}(X)[-1] \rangle. 
\end{align}
Here $\Coh_{\le 1}(X)$ is the category of 
sheaves $F$ with $\dim \Supp(F) \le 1$. 
Let $\overline{\mu}_{\omega}$ be the slope function on $\Coh_{\le 1}(X)$
defined by
\begin{align*}
\overline{\mu}_{\omega}(F) \cneq 
\frac{\ch_3(F)}{\ch_2(F) \cdot \omega}. 
\end{align*}
Similarly to the $\mu_{\omega}$-stability, 
the above slope function on $\Coh_{\le 1}(X)$
defines the $\overline{\mu}_{\omega}$-stability 
on $\Coh_{\le 1}(X)$. 
For any interval $I\subset \mathbb{R} \cup \{\infty\}$, we set
\begin{align*}
\cC_{I} \cneq 
\left\langle F \in \Coh_{\le 1}(X): 
\begin{array}{c}
F \mbox{ is } \overline{\mu}_{\omega} \mbox{-semistable} \\
\mbox{ with } 
\overline{\mu}_{\omega}(F) \in I 
\end{array}
\right\rangle[-1] \cup\{0\}. 
\end{align*}
Then 
using the notation in Definition~\ref{def:nest}, 
we can write (\ref{Bu}) 
as
\begin{align}\label{def:Bu}
\bB_{\mu}=\langle \Coh_{\mu}(X), \cC_{\infty}, \cC_{[0, \infty)}, \cC_{< 0} 
\rangle. 
\end{align}
Note that $\cC_{\infty}$ consists of $Q[-1]$ for zero dimensional sheaves $Q$. 
We will consider its subcategory
\begin{align}\label{Cu}
\dD_{\mu} \cneq 
\langle \Coh_{\mu}(X), \cC_{\infty}, \cC_{[0, \infty)} \rangle. 
\end{align}
\begin{lem}\label{closed}
The subcategories $\dD_{\mu}$, 
$\langle \Coh_{\mu}(X), \cC_{\infty} \rangle$, 
$\cC_{\infty}$, $\cC_{[0, \infty]}$
in $\bB_{\mu}$
are closed under quotients in the abelian category $\bB_{\mu}$. 
\end{lem}
\begin{proof}
The categories $\dD_{\mu}$, $\langle \Coh_{\mu}(X), \cC_{\infty} \rangle$
are closed 
under quotients since they are torsion parts of some torsion 
pairs of $\bB_{\mu}$. 
Since $E \in \bB_{\mu}$ 
has $\rank(E) \ge 0$ with $\rank(E)=0$ if 
and only if $E \in \Coh_{\le 1}(X)[-1]$, 
we see that $\Coh_{\le 1}(X)[-1]$ is closed under 
quotients and subobjects in $\bB_{\mu}$. 
Then the categories $\cC_{\infty}$, $\cC_{[0, \infty]}$
are also closed under quotients in $\bB_{\mu}$ since
they are also torsion parts of some torsion pairs 
of $\Coh_{\le 1}(X)[-1]$. 
\end{proof}
\begin{rmk}\label{rmk:Du}
Note that both of $\mu_{\omega}$-semistable objects, 
PT-semistable objects
with Chern character (\ref{rdbn}) are
objects in $\dD_{\mu}$. 
Indeed, they are contained in the smaller 
subcategories $\Coh_{\mu}(X)$, 
$\langle \Coh_{\mu}(X), \cC_{\infty} \rangle$
respectively. 
\end{rmk}
\subsection{Completions of Hall algebras}
It is known that the stack of objects in 
$\aA_{\mu}$ forms an open substack of $\mM$
(cf.~\cite[Proposition~4.11]{PiYT}).
Therefore by Subsection~\ref{subsec:poisson}, 
we can define the 
Hall algebra $H(\aA_{\mu})$
of $\aA_{\mu}$, 
the associated Poisson algebra
$H^{\rm{sc}}(\aA_{\mu})$, 
and the Poisson algebra homomorphism
\begin{align}\label{funct:I}
I \colon H^{\rm{sc}}(\aA_{\mu}) \to C(X). 
\end{align}
We construct 
certain completions of 
the above Hall algebras, using 
some inequalities of Chern characters. 
First for any
torsion free $\mu_{\omega}$-semistable 
sheaf $E$ on $X$, we have the 
Bogomolov inequality
\begin{align*}
(\ch_1(E) \omega^2)^2 \ge 2\ch_0(E)\omega^3 \cdot \ch_2(E) \omega. 
\end{align*} 
Second by Langer~\cite[Section~3]{Langer2},
there is a function 
\begin{align*}
l \colon H^0(X) \oplus H^2(X) \oplus H^4(X) \to \mathbb{Q}_{>0}
\end{align*} 
such that any torsion free 
$\mu_{\omega}$-semistable sheaf 
$E$ satisfy 
\begin{align*}
\ch_3(E) \le l(\ch_0(E), \ch_1(E), \ch_2(E)). 
\end{align*}
For a fixed $(r, D) \in H^0(X) \oplus H^2(X)$
satisfying (\ref{coprime}), we define
\begin{align*}
\Gamma_{r, D} \cneq 
\left\{(r, D, -\beta, -n) \in \Gamma : 
\omega \beta \ge -\frac{(D \omega^2)^2}{2r \omega^3}
, n\ge -l(r, D, -\beta)\right\}. 
\end{align*}
We also define
\begin{align*}
\Gamma_{\sharp} \cneq \{(0, 0, -\beta, -n) \in \Gamma : 
\beta\ge 0, n\ge 0\}. 
\end{align*}
Here $\beta>0$ means that it is the Poincare dual
of an effective algebraic one cycle on $X$. 
We have the following 
obvious lemma: 
\begin{lem}\label{lem:fin}
\begin{enumerate}
\item For any $E \in \dD_{\mu}$ with 
$(\ch_0(E), \ch_1(E))=(r, D)$, we have 
$\ch(E) \in \Gamma_{r, D}$. 
\item For any $F \in \dD_{\mu}$ with 
$\ch_0(F)=0$, we have $\ch(F) \in \Gamma_{\sharp}$. 
\item
For $v \in \Gamma_{r, D}$ and $v' \in \Gamma_{\sharp}$, 
we have $v+v' \in \Gamma_{r, D}$. 
\item For $v \in \Gamma_{\sharp}$, there
is only a finite number of ways to write 
it $v_1+v_2+ \cdots +v_l$
for $v_1, \cdots, v_l \in \Gamma_{\sharp} \setminus \{0\}$. 
\item 
For $v \in \Gamma_{r, D}$, there is only a finite 
number of ways to write it 
$v=v_1+v_2+ \cdots v_l+v_{l+1}$
for $v_1, \cdots, v_l \in \Gamma_{\sharp} \setminus \{0\}$
and $v_{l+1} \in \Gamma_{r, D}$. 
\end{enumerate}
\end{lem}
The above lemma
 also imply the following lemma,
which will be used in Subseciton~\ref{subsec:L}:  
\begin{lem}\label{lem:bound}
The set of objects in $\dD_{\mu}$ with a fixed 
Chern character is bounded. 
\end{lem}
\begin{proof}
The result follows from Lemma~\ref{lem:fin}
together with the fact that the set of 
semistable sheaves with a fixed Chern 
character is bounded. 
\end{proof}
We set
\begin{align*}
\widehat{H}_{r, D}(\aA_{\mu}) \cneq 
\prod_{v \in \Gamma_{r, D}} H_v(\aA_{\mu}), \
\widehat{H}_{\sharp}(\aA_{\mu}) \cneq \prod_{v \in \Gamma_{\sharp}}H_{v}(\aA_{\mu}). 
\end{align*}
By Lemma~\ref{lem:fin}, 
the Hall product on $H(\aA_{\mu})$ induces the 
one on $\widehat{H}_{\sharp}(\aA_{\mu})$, 
and $\widehat{H}_{r, D}(\aA_{\mu})$
is a bi-module over $\widehat{H}_{\sharp}(\aA_{\mu})$. 
Similarly to 
Subsection~\ref{subsec:poisson}, 
we can define the subspaces of regular elements
\begin{align*}
\widehat{H}_{r, D}^{\rm{reg}}(\aA_{\mu})
\subset \widehat{H}_{r, D}(\aA_{\mu}), \ 
\widehat{H}_{\sharp}^{\rm{reg}}(\aA_{\mu})
\subset \widehat{H}_{\sharp}(\aA_{\mu})
\end{align*}
such that 
$\widehat{H}_{\sharp}^{\rm{reg}}(\aA_{\mu})$
is a subalgebra of 
$\widehat{H}_{\sharp}(\aA_{\mu})$, 
and $\widehat{H}_{r, D}^{\rm{reg}}(\aA_{\mu})$
is a bi-module over 
$\widehat{H}_{\sharp}^{\rm{reg}}(\aA_{\mu})$. 
Also we can define the quotient spaces
\begin{align*}
\widehat{H}_{r, D}^{\rm{sc}}(\aA_{\mu})
&=\widehat{H}_{r, D}^{\rm{reg}}(\aA_{\mu})/(\mathbb{L}-1) \cdot 
\widehat{H}_{r, D}^{\rm{reg}}(\aA_{\mu}) \\
\widehat{H}_{\sharp}^{\rm{sc}}(\aA_{\mu})
&=\widehat{H}_{\sharp}^{\rm{reg}}(\aA_{\mu})/(\mathbb{L}-1) \cdot 
\widehat{H}_{\sharp}^{\rm{reg}}(\aA_{\mu})
\end{align*}
such that we have 
the induced Poisson algebra structure 
on $\widehat{H}_{\sharp}^{\rm{sc}}(\aA_{\mu})$, 
and 
$\widehat{H}_{r, D}^{\rm{sc}}(\aA_{\mu})$
is a Poisson bi-module over $\widehat{H}_{\sharp}^{\rm{sc}}(\aA_{\mu})$.

\begin{rmk}
For $F \in \aA_{\mu}$, it is easy to see that 
$\ch(F) \in \Gamma_{\sharp}$ if and only if $F \in \Coh_{\le 1}(X)[-1]$. 
Hence we can also write $\widehat{H}_{\sharp}(\aA_{\mu})$ as
\begin{align*}
\widehat{H}_{\sharp}(\aA_{\mu})=\prod_{v\in \Gamma_{\sharp}}
H_v(\Coh_{\le 1}(X)[-1]). 
\end{align*}
\end{rmk}
By Lemma~\ref{lem:fin} again, 
for any $\gamma \in \widehat{H}_{\sharp}(\aA_{\mu})$ with zero 
$H_0(\aA_{\mu})$-component, we have the well-defined 
elements
\begin{align}\label{well-def}
\exp(\gamma),  \ \log(1+\gamma), \ (1+\gamma)^{-1} \in 
\widehat{H}_{\sharp}(\aA_{\mu}). 
\end{align}
The following is the important consequence of 
Joyce's absence of pole result: 
\begin{thm}\emph{(\cite[Theorem~8.7]{Joy3}, \cite[Theorem~6.3]{BrH})}\label{thm:pole}
For $\gamma \in \widehat{H}_{\sharp}(\aA_{\mu})$
with zero $H_0(\aA_{\mu})$-component, we have
\begin{align*}
(\mathbb{L}-1) \cdot \log(1+\gamma) \in \widehat{H}_{\sharp}^{\rm{reg}}
(\aA_{\mu}). 
\end{align*}
\end{thm}
We also set
\begin{align*}
\widehat{C}_{r, D}(X) \cneq 
\prod_{v \in \Gamma_{r, D}} C_v(X), \
\widehat{C}_{\sharp}(X) \cneq \prod_{v \in \Gamma_{\sharp}}C_{v}(X). 
\end{align*}
Similarly as above, 
$\widehat{C}_{\sharp}(X)$
is a Poisson algebra
and 
 $\widehat{C}_{r, D}(X)$ is a Poisson bi-module 
over $\widehat{C}_{\sharp}(X)$. 
The integration map (\ref{funct:I})
induces the maps
\begin{align}\label{integ}
I_{r, D} \colon 
\widehat{H}_{r, D}(\aA_{\mu}) \to \widehat{C}_{r, D}(X), \ 
I_{\sharp} \colon \widehat{H}_{\sharp}(\aA_{\mu}) \to 
\widehat{C}_{\sharp}(X)
\end{align}
such that $I_{\sharp}$ 
is a Poisson algebra homomorphism 
and $I_{r, D}$ is a Poisson bi-module homomorphism
over $\widehat{H}_{\sharp}(\aA_{\mu})$. 

\subsection{Elements of the Hall algebra}
Let 
\begin{align}\label{stack:M}
\mM_{\DT}(r, D), \ \mM_{\PT}(r, D)
\end{align}
be
the stacks of $\mu_{\omega}$-stable sheaves, 
PT-stable objects, respectively. 
\begin{lem}\label{lem:gerbe}
The stacks (\ref{stack:M}) are
 $\mathbb{C}^{\ast}$-gerbes
over the unions of 
(\ref{MDT}), (\ref{MPT}) 
for all possible $(\beta, n)$
respectively.
\end{lem}
\begin{proof}
Since any $[E] \in \mM_{\DT}(r, D)$
is a stable sheaf by the 
coprime condition (\ref{coprime}), 
we have $\Aut(E)=\mathbb{C}^{\ast}$. 
Similarly any $I^{\bullet} \in \mM_{\PT}(r, D)$
is a stable object with respect to a certain 
polynomial stability condition~\cite{JLo},
we also have 
$\Aut(I^{\bullet})=\mathbb{C}^{\ast}$. 
Hence the lemma follows. 
\end{proof} 
The above lemma in particular implies that 
each connected component 
of (\ref{stack:M}) are of finite type. 
Also by Remark~\ref{rmk:Du}
and Lemma~\ref{lem:fin} (i),
the Chern characters of  
objects in (\ref{stack:M}) are
 contained in $\Gamma_{r, D}$.  
Therefore for 
$\star \in \{\rm{DT}, \rm{PT}\}$, we obtain the elements
\begin{align}\label{delta:star}
\delta_{\star}(r, D) \cneq
[\mM_{\star}(r, D) \to \oO bj(\aA_{\mu})] \in 
\widehat{H}_{r, D}(\aA_{\mu}).
\end{align}
We also set 
\begin{align*}
\overline{\delta}_{\star}(r, D) \cneq 
(\mathbb{L}-1) \cdot \delta_{\star}(r, D). 
\end{align*}
Then Lemma~\ref{lem:gerbe}
implies that 
\begin{align}\label{delta:star}
\overline{\delta}_{\star}(r, D)
=\sum_{v\in \Gamma_{r, D}} [M_{\star}(v) \to \oO bj(\aA_{\mu})] 
\in \widehat{H}^{\rm{reg}}_{r, D}(\aA_{\mu}). 
\end{align}
Here $M_{\star}(v)$ are the moduli spaces
 (\ref{MDT}) for $\star=\DT$, 
(\ref{MPT}) for $\star=\PT$ respectively.

In the similar way, 
for any interval $I \subset \mathbb{R}_{\ge 0} \cup \{\infty\}$
the stack 
of objects 
$\oO bj(\cC_I)$ in $\cC_{I}$ decomposes into
\begin{align*}
\oO bj(\cC_I)=\coprod_{v \in \Gamma_{\sharp}, \overline{\mu}_{\omega}(v)\in I,  v\neq 0}
 \oO bj_v(\cC_I)
\end{align*}
such that each component $\oO bj_v(\cC_I)$ is a finite type stack. 
Here we have used the following map
\begin{align*}
\overline{\mu}_{\omega} \colon 
\Gamma_{\sharp} \setminus \{0\} \to \mathbb{Q}_{\ge 0} \cup \{\infty\}, \ 
(\beta, n) \mapsto \frac{n}{\omega \cdot \beta}. 
\end{align*}
Hence we have the element
\begin{align*}
\delta(\cC_{I})
\cneq [
\oO bj(\cC_{I}) \to \oO bj(\aA_{\mu})] \in 
\widehat{H}_{\sharp}(\aA_{\mu}). 
\end{align*}
Applying (\ref{well-def}), we obtain the element
\begin{align}\label{epsilon}
\epsilon(\cC_I)
 \cneq \log(\delta(\cC_I))
\in \widehat{H}_{\sharp}(\aA_{\mu}). 
\end{align}
By Theorem~\ref{thm:pole}, we obtain the element: 
\begin{align}\label{ep:bar}
\overline{\epsilon}(\cC_I)
\cneq (\mathbb{L}-1) \cdot 
\epsilon(\cC_I) \in \widehat{H}^{\rm{reg}}_{\sharp}(\aA_{\mu}).
\end{align}

\subsection{Applications of the integration map}
Let us project 
elements (\ref{delta:star}) 
to $\widehat{H}^{\rm{sc}}_{r, D}(\aA_{\mu})$ and 
apply the integration map (\ref{integ}). 
Then 
 we have 
the following relations:
\begin{align}\label{rel:DT}
&I_{r, D}\left(\overline{\delta}_{\DT}(r, D) \right)
=-\sum_{v \in \Gamma_{r, D}} \DT(v) \cdot c_{v} \\
\notag
&I_{r, D}\left(\overline{\delta}_{\PT}(r, D) \right)
=-\sum_{v \in \Gamma_{r, D}} \PT(v) \cdot c_{v}. 
\end{align}
Here we have used the Behrend's result~\cite{Beh}
describing virtual classes associated with 
symmetric obstruction theories by 
the integrations of 
his constructible functions.  

Next we project elements 
(\ref{ep:bar}) to 
$\widehat{H}^{\rm{sc}}_{\sharp}(\aA_{\mu})$ and 
apply the integration map (\ref{integ}). 
For 
$\overline{\mu} \in \mathbb{Q}_{\ge 0} \cup \{\infty\}$ and a
non-zero $v \in \Gamma_{\sharp}$
with $\overline{\mu}_{\omega}(v)=\overline{\mu}$, 
we obtain 
the invariant $N_{v} \in \mathbb{Q}$
given by the following formula: 
\begin{align}\label{def:N}
I_{\sharp}(\overline{\epsilon}(\cC_{\overline{\mu}}))
=-\sum_{0\neq v \in \Gamma_{\sharp}, \overline{\mu}_{\omega}(v)=\overline{\mu}}
N_{v} \cdot c_v. 
\end{align}
Below for $(0, 0, -\beta, -n) \in \Gamma_{\sharp}$, we write
\begin{align}\label{Nnb}
N_{n, \beta} \cneq N_{(0, 0, -\beta, -n)} \in \mathbb{Q}. 
\end{align}
\begin{rmk}\label{rmk:N}
The invariant $N_{n, \beta}$ is
a virtual count of $\overline{\mu}_{\omega}$-semistable 
sheaves $F \in \Coh_{\le 1}(X)$
with $[F]=\beta$, $\chi(F)=n$, 
which played an important role in 
the wall-crossing of curve 
counting theory, e.g.~\cite[Definition~4.7]{Tsurvey}. 
In particular if $N_{n, \beta} \neq 0$, 
then $\beta$ is either zero or 
a Poincar\'e dual of an effective one cycle on $X$, 
and $n\in \mathbb{Z}_{\ge 0}$. 
\end{rmk}
The following fact is also standard, 
but we include the proof for completeness:
\begin{lem}\label{standard}
For any interval $I \subset \mathbb{R}_{\ge 0} \cup \{\infty\}$, we have
\begin{align*}
I_{\sharp}(\overline{\epsilon}(\cC_{I}))
=-\sum_{0\neq v \in \Gamma_{\sharp}, \overline{\mu}_{\omega}(v) \in I}
N_{v} \cdot c_v.
\end{align*}
\end{lem}
\begin{proof}
By the
the existence of Harder-Narasimhan filtrations with respect to the 
$\overline{\mu}_{\omega}$-stability, we have 
the identity in $\widehat{H}_{\sharp}(\aA_{\mu})$
(cf.~\cite[Theorem~5.11]{Joy4})
\begin{align*}
\delta(\cC_I)
=\prod^{\longrightarrow}_{\overline{\mu} \in I} \delta(\cC_{\overline{\mu}}). 
\end{align*} 
In the RHS, we take the product with the decreasing order 
of $\overline{\mu}$. 
By taking the logarithm of both sides
and multiplying $(\mathbb{L}-1)$, we obtain 
the identity in $\widehat{H}_{\sharp}^{\rm{sc}}(\aA_{\mu})$: 
\begin{align*}
\overline{\epsilon}(\cC_I)
=\sum_{\overline{\mu} \in I} \overline{\epsilon}(\cC_{\overline{\mu}})
+ \left\{ \mbox{Nested Poisson brackets in }
\overline{\epsilon}(\cC_{\overline{\mu}})  \right\}. 
\end{align*}
Since $\chi(v_1, v_2)=0$
for $v_i \in \Gamma_{\sharp}$, 
the property of the integration map
$I_{\sharp}$ shows that 
\begin{align*}
I_{\sharp}(\overline{\epsilon}(\cC_I))=
\sum_{\overline{\mu} \in I}
I_{\sharp}(\overline{\epsilon}(\cC_{\overline{\mu}})). 
\end{align*} 
By (\ref{def:N}),
we obtain the desired identity.  
\end{proof}

\subsection{Proof of Theorem~\ref{thm:higherDTPT}}
We denote by 
\begin{align*}
\Coh^{P}_{\mu}(X)
\subset D^b \Coh(X)
\end{align*}
 the 
category of PT-semistable objects 
$I^{\bullet}$ with $\mu_{\omega}(I^{\bullet})=\mu$. 
Let $\bB_{\mu}$ be the abelian 
category introduced in (\ref{def:Bu}). 
\begin{lem}\label{lem:cid}
We have the following identity in $\bB_{\mu}$: 
\begin{align}\label{DTPT:id}
\langle \Coh_{\mu}(X), \cC_{\infty} \rangle
= \langle \cC_{\infty}, \Coh^{P}_{\mu}(X) \rangle. 
\end{align}
\end{lem}
\begin{proof}
By the definition of PT semistable objects, 
we have
\begin{align*}
\Coh^{P}_{\mu}(X)=\{ E \in 
\langle \Coh_{\mu}(X), \cC_{\infty} \rangle
 : \Hom(\cC_{\infty}, E)=0\}.  
\end{align*}
It is enough to show that any object 
$E$ in the LHS of (\ref{DTPT:id})
fits into an exact sequence 
\begin{align}\label{TEF}
0 \to T \to E \to F \to 0
\end{align}
in  
$\bB_{\mu}$
with $T \in \cC_{\infty}$
and $F \in \Coh^P_{\mu}(X)$. 
Suppose that $E \notin \Coh^P_{\mu}(X)$. 
Then there is a non-zero morphism 
$T \to E$ for some
$T \in \cC_{\infty}$. 
By Lemma~\ref{closed}, 
$\cC_{\infty}$ is closed under 
quotients in $\bB_{\mu}$, 
so we may assume that 
$T \to E$ is injective in 
$\bB_{\mu}$. Then 
the quotient 
$F=E/T$ in $\bB_{\mu}$
is an object in the LHS of (\ref{DTPT:id})
by Lemma~\ref{closed} again.  
Because $\aA_{\mu}$ is noetherian 
(cf.~the proof of~\cite[Lemma~3.2.4]{BMT}), 
the abelian category
$\bB_{\mu}$ is also noetherian, 
so this process must terminate.  
Hence we obtain the exact sequence (\ref{TEF}). 
\end{proof}
Let us consider the series (\ref{DT:series}), (\ref{PT:series}). 
We are now ready to prove 
Theorem~\ref{thm:higherDTPT}: 
\begin{thm}\label{thm:DTPTseries}
We have the following identity: 
\begin{align*}
\DT_{r, D}(q, t)=M((-1)^r q)^{r \cdot e(X)} \cdot \PT_{r, D}(q, t). 
\end{align*}
\end{thm}
\begin{proof}
By Lemma~\ref{lem:cid},
we have the following identity 
in $\widehat{H}_{r, D}(\aA_{\mu})$: 
\begin{align*}
\delta_{\DT}(r, D)\ast \delta(\cC_{\infty})
=\delta(\cC_{\infty}) \ast \delta_{\PT}(r, D). 
\end{align*} 
Using (\ref{epsilon}),
we obtain the identity:  
\begin{align}\label{del:DTPT}
\delta_{\DT}(r, D) =
\exp(\epsilon(\cC_{\infty}))
\ast \delta_{\PT}(r, D) \ast \exp(-\epsilon(\cC_{\infty})). 
\end{align}
For $a \in \widehat{H}_{\sharp}(\aA_{\mu})$, its 
adjoint action on 
$\widehat{H}_{r, D}(\aA_{\mu})$ is denoted by
\begin{align*}
\mathrm{Ad}(a) \circ x  \cneq
a \ast x -x \ast a \colon \widehat{H}_{r, D}(\aA_{\mu})
\to \widehat{H}_{r, D}(\aA_{\mu}). 
\end{align*}
Then applying the 
Baker-Campbell-Hausdorff formula to (\ref{del:DTPT}), 
we obtain 
\begin{align*}
\delta_{\DT}(r, D) =
\exp\left(\mathrm{Ad}(\epsilon(\cC_{\infty}))
 \right) \circ 
\delta_{\PT}(r, D).
\end{align*}
We multiply $(\mathbb{L}-1)$ to both sides, 
and project them to $\widehat{H}_{r, D}^{\rm{sc}}(\aA_{\mu})$. 
By writing the adjoint action 
of $a \in \widehat{H}_{\sharp}^{\rm{sc}}(\aA_{\mu})$ as
\begin{align*}
\mathrm{Ad}^{\rm{sc}}(a) \circ x \cneq \{a, x\} \colon 
\widehat{H}_{r, D}^{\rm{sc}}(\aA_{\mu})
\to \widehat{H}_{r, D}^{\rm{sc}}(\aA_{\mu})
\end{align*}
we obtain the following identity in 
$\widehat{H}_{r, D}^{\rm{sc}}(\aA_{\mu})$: 
\begin{align}\label{de:id}
\overline{\delta}_{\DT}(r, D) 
=\exp\left(\mathrm{Ad}^{\rm{sc}}(\overline{\epsilon}(\cC_{\infty}))
 \right) \circ 
\overline{\delta}_{\PT}(r, D).
\end{align}
By applying the integration map $I_{r, D}$ in (\ref{integ}), 
using (\ref{def:N}) and the computation (\ref{RR}),
we obtain the formula
\begin{align*}
\DT_{r, D}(q, t)=
\exp\left( \sum_{n>0} (-1)^{rn-1}rn N_{n, 0}q^n \right)
\cdot \PT_{r, D}(q, t). 
\end{align*}
By substituting $(r, D)=(1, 0)$, $t=0$, and 
comparing the computation 
of the series 
$\DT_{(1, 0)}(q, 0)$ in~\cite{BBr}, \cite{LP}, \cite{Li}, 
we obtain 
\begin{align*}
\exp\left( \sum_{n>0} (-1)^{n-1}n N_{n, 0}q^n \right)
=M(-q)^{e(X)}. 
\end{align*}
Therefore we obtain the desired identity. 
\end{proof}

\section{Rationality of higher rank PT invariants}
In this section, we prove Theorem~\ref{intro:rational}. 
\subsection{Category $\Coh_{\mu}^{L}(X)$}
As before, let $\mu$ be the rational number (\ref{def:mu}), 
and $\dD_{\mu}$ be the category~(\ref{Cu}). 
We define the following category
\begin{align*}
\Coh^{L}_{\mu}(X) \cneq 
\{ E \in \dD_{\mu} : \Hom(F, E)=0 \mbox{ for any }
F \in \cC_{[0, \infty]}\}. 
\end{align*}
We prepare some lemmas on the above category. 
\begin{lem}\label{L:id}
We have the following identity: 
\begin{align*}
\dD_{\mu}=\langle \cC_{[0, \infty]}, \Coh^L_{\mu}(X) \rangle. 
\end{align*}
\end{lem}
\begin{proof}
Using Lemma~\ref{closed} and replacing
$\langle \Coh_{\mu}(X), \cC_{\infty} \rangle$,  
$\cC_{\infty}$, $\Coh_{\mu}^{P}(X)$
by $\dD_{\mu}$, $\cC_{[0, \infty]}$, $\Coh_{\mu}^L(X)$
in the proof of Lemma~\ref{lem:cid}, 
we obtain the desired identity. 
\end{proof}
\begin{lem}\label{lem:aut}
For any $E \in \Coh^L_{\mu}(X)$ with
$(\ch_0(E), \ch_1(E))=(r, D)$, we have 
$\Aut(E)=\mathbb{C}^{\ast}$. 
\end{lem}
\begin{proof}
It is enough to show that any non-zero 
morphism $\phi \colon E \to E$ is 
an isomorphism. 
Let $0\neq F \subset E$ be the image of 
$\phi$ in $\bB_{\mu}$, 
and $G \subset F$ the kernel of $\phi$
in $\bB_{\mu}$. 
Suppose that $F \subsetneq E$.
Because of the coprime condition (\ref{coprime}), 
we have either $\rank(F)=0$ or $\rank(G)=0$. 
In the $\rank(F)=0$ case, 
we have $F \in \Coh_{\le 1}(X)[-1]$. 
On the other hand, 
we have $F \in \dD_{\mu}$ by Lemma~\ref{closed}, 
 hence $F \in \cC_{[0, \infty]}$. 
But this contradicts to  
$E\in \Coh^L_{\mu}(X)$. 
In the $\rank(G)=0$ case, 
the condition $E \in \Coh^L_{\mu}(X)$ implies that
$G \in \cC_{<0}$. 
Also by Lemma~\ref{closed},
we have $E/F \in \cC_{[0, \infty]}$. 
Since $\ch(G)=\ch(E/F)$, 
we have the contradiction. 
Therefore we must have $F=E$. 
But then $G \in \bB_{\mu}$
satisfies $\ch(G)=0$, hence 
$G=0$. 
Therefore $\phi$ is an isomorphism.  
\end{proof}
By Lemma~\ref{L:id}, the abelian 
category $\bB_{\mu}$ is written as
\begin{align}\label{decom:L}
\bB_{\mu}=\langle \cC_{[0, \infty]}, \Coh_{\mu}^{L}(X), \cC_{<0} \rangle. 
\end{align}
Let $\mathbb{D}$ be the dualizing functor
on $D^b \Coh(X)$
defined by 
\begin{align*}
\mathbb{D}(E) \cneq \dR \hH om(E, \oO_X). 
\end{align*}
The next proposition shows that the category 
$\Coh_{\mu}^{L}(X)$ behaves well under the duality:  
\begin{prop}\label{prop:decom}
The abelian category $\bB_{-\mu}$ is written as
\begin{align}\label{decom:Bu}
\bB_{-\mu}=
\langle \cC_{(0, \infty]}, \mathbb{D}(\Coh_{\mu}^{L}(X)), 
\cC_{\le 0} \rangle. 
\end{align}
\end{prop}
\begin{proof}
Note that $\bB_{\mu}$ is also written as
\begin{align*}
\bB_{\mu}=\langle \cC_{\infty}, \Coh_{\mu}^{P}(X), 
\cC_{<\infty} \rangle
\end{align*}
by Lemma~\ref{lem:cid}. 
Let $\bB_{\mu}^{\dag}$ be the tilting 
at $\cC_{\infty}$, i.e.
\begin{align*}
\bB_{\mu}^{\dag}
\cneq \langle \Coh_{\mu}^{P}(X), 
\cC_{<\infty}, \cC_{\infty}[-1] \rangle. 
\end{align*} 
We first claim that $\mathbb{D}$ induces the equivalence
\begin{align}\label{B:id}
\mathbb{D} \colon \bB_{\mu} \stackrel{\sim}{\to}
\bB_{-\mu}^{\dag}. 
\end{align}
Indeed 
we have $\mathbb{D}(\cC_{<\infty})=\cC_{<\infty}$ 
and $\mathbb{D}(\cC_{\infty})=\cC_{\infty}[-1]$. 
Also by~\cite[Lemma~4.16]{PiYT}, the following holds:
\begin{align*}
\mathbb{D}(\Coh_{\mu}^{P}(X))
=\{ E \in \bB_{-\mu} \colon 
\Hom(\Coh_{\le 1}(X)[-1], E)=0\}. 
\end{align*}
In particular, we have 
$\mathbb{D}(\Coh_{\mu}^P(X)) \subset
\bB_{-\mu}^{\dag}$, 
therefore $\mathbb{D}(\bB_{\mu}) \subset 
\bB_{-\mu}^{\dag}$
holds. 
Similar argument shows that 
$\mathbb{D}(\bB_{-\mu}^{\dag}) \subset \bB_{\mu}$, 
hence the equivalence (\ref{B:id})
holds.  
Applying (\ref{B:id}) to the description (\ref{decom:L})
and noting that $\mathbb{D}(\cC_{[0, \infty)})=\cC_{\le 0}$, 
$\mathbb{D}(\cC_{<0})=\cC_{(0, \infty)}$, we obtain the identity
\begin{align*}
\bB_{-\mu}^{\dag}
=\langle \cC_{(0, \infty)}, \mathbb{D}(\Coh_{\mu}^L(X)), \cC_{\le 0}, \cC_{\infty}[-1] \rangle. 
\end{align*}
By tilting $\cC_{\infty}[-1]$ back, 
we obtain the desired identity (\ref{decom:Bu}). 
\end{proof}

The following corollary obviously follows from 
Proposition~\ref{prop:decom}: 
\begin{cor}\label{cor:inc}
We have $\mathbb{D}(\Coh_{\mu}^{L}(X)) \subset \dD_{-\mu}$. 
\end{cor}
\subsection{$L$-invariants}\label{subsec:L}
Let
\begin{align}\label{ML}
\mM_{\rm{L}}(r, D) \subset \oO bj(\aA_{\mu})
\end{align}
be the substack of objects $E \in \Coh_{\mu}^L(X)$
with $(\ch_0(E), \ch_1(E))=(r, d)$. 
We will not pursue to prove that (\ref{ML}) is an
open substack of finite type. 
Instead, the
definition of $\Coh_{\mu}^{L}(X)$
and 
Lemma~\ref{lem:bound} 
easily imply that the 
$\mathbb{C}$-valued points of the 
stack (\ref{ML})
on each component $\oO bj_{v}(\aA_{\mu})$ 
forms a constructible subset. 
Therefore by using the motivic relation, 
we obtain the element
\begin{align*}
\delta_{\rm{L}}(r, D) \cneq [\mM_{\rm{L}}(r, D) \subset \oO bj(\aA_{\mu})] 
\in \widehat{H}_{r, D}(\aA_{\mu}). 
\end{align*}
Also Lemma~\ref{lem:aut}
implies that
\begin{align*}
\overline{\delta}_{\rm{L}}(r, D) \cneq (\mathbb{L}-1) \cdot 
\delta_{\rm{L}}(r, D) 
\in \widehat{H}_{r, D}^{\rm{reg}}(\aA_{\mu}). 
\end{align*} 
By projecting it
to  
$\widehat{H}_{r, D}^{\rm{sc}}(\aA_{\mu})$
and applying (\ref{integ}), 
we obtain the integer valued invariants
\begin{align*}
\mathrm{L}(r, D, -\beta, -n) \in \mathbb{Z}
\end{align*}
by the formula
\begin{align*}
I_{r, D}(\overline{\delta}_{\rm{L}}(r, D))
=-\sum_{v\in \Gamma_{r, D}} \mathrm{L}(v) \cdot c_{v}. 
\end{align*}

\subsection{Properties of the generating series}
We discuss some properties of the
generating series of $N_{n, \beta}$ and 
$\mathrm{L}(r, D, -\beta, -n)$. 
For the former invariants, we have the following:
\begin{lem}\label{lem:Nrational}
For a fixed $\beta>0$, the generating series
\begin{align}\label{N:rational}
\sum_{n \ge 0} N_{n, \beta}q^n, \ 
\sum_{n\ge 0} n N_{n, \beta}q^n
\end{align}
are rational functions in $q$. 
\end{lem}
\begin{proof}
Since the $\overline{\mu}_{\omega}$-stability 
is preserved under taking the 
tensor product with a line bundle $\lL$
satisfying $c_1(\lL)=\omega$, we have 
$N_{n, \beta}=N_{n+\beta \cdot \omega, \beta}$. 
By noting Remark~\ref{rmk:N} and $\beta \cdot \omega>0$, 
the desired property of (\ref{N:rational}) is an 
easy consequence of this periodicity. 
\end{proof}
\begin{rmk}
In fact $N_{n, \beta}$ further satisfies
$N_{n, \beta}=N_{-n, \beta}$
and this fact also shows that 
the second series in (\ref{N:rational}) 
is a rational function invariant under
$q\leftrightarrow 1/q$, 
as proved in~\cite[Lemma~4.6]{Tolim2}. 
However the first series in (\ref{N:rational}) 
is not invariant under $q \leftrightarrow 1/q$. 
\end{rmk}
As for the L-invariants, we have the following: 
\begin{lem}\label{lem:Lpoly}
For a fixed $(r, D, \beta)$, the 
generating series
\begin{align}\label{series:Ln}
\sum_{6n \in \mathbb{Z}} \mathrm{L}(r, D, -\beta, -n)q^n
\end{align}
is a Laurent polynomial in $q^{\frac{1}{6}}$. 
\end{lem}
\begin{proof}
Suppose that $\mathrm{L}(r, D, -\beta, -n)\neq 0$. 
Then by Lemma~\ref{lem:fin} (i)
and Corollary~\ref{cor:inc},
we have
 $(r, \pm D, -\beta, \mp n) \in \Gamma_{r, \pm D}$. 
Hence
the definition of $\Gamma_{r, D}$
implies that 
$\mathrm{L}(r, D, -\beta, -n)=0$ for 
a fixed $(r, D, \beta)$ and 
$\lvert n \rvert \gg 0$. 
This implies that the series (\ref{series:Ln})
is a Laurent polynomial in $q^{\frac{1}{6}}$. 
\end{proof}

\subsection{Proof of rationality}
We define the following generating series
\begin{align*}
\mathrm{L}_{r, D}(q, t)
\cneq \sum_{\beta, n}
\mathrm{L}(r, D, -\beta, -n)q^n t^{\beta}. 
\end{align*}
The following result is a  
higher rank version of the 
product expansion formula given in~\cite{Tolim2}, \cite{BrH}: 
\begin{thm}\label{thm:PT=L}
We have the following product expansion formula: 
\begin{align*}
\PT_{r, D}(q, t)&=\mathrm{L}_{r, D}(q, t) \\
&\cdot 
\prod_{n\ge 0, \beta>0}
\exp\left(
N_{n, \beta}\{(-1)^r q\}^n \{(-1)^D t\}^{\beta} \right)^{D\beta-rn}. 
\end{align*}
\end{thm}
\begin{proof}
The results of 
Lemma~\ref{lem:cid} and Lemma~\ref{L:id}
imply 
\begin{align*}
\langle \Coh_{\mu}^P(X), \cC_{[0, \infty)} \rangle
=\langle \cC_{[0, \infty)}, \Coh_{\mu}^{\mathrm{L}}(X) \rangle. 
\end{align*}
We obtain the following identity in 
$\widehat{H}_{r, D}(\aA_{\mu})$
\begin{align*}
\delta_{\PT}(r, D) \ast \delta(\cC_{[0, \infty)})
=\delta(\cC_{[0, \infty)}) \ast
\delta_{\mathrm{L}}(r, D). 
\end{align*}
Therefore by the same argument deducing (\ref{de:id}) 
in the proof of Theorem~\ref{thm:DTPTseries},  
we obtain the following identity in 
$\widehat{H}_{r, D}^{\rm{sc}}(\aA_{\mu})$: 
\begin{align}\notag
\overline{\delta}_{\PT}(r, D) 
=\exp\left(\mathrm{Ad}^{\rm{sc}}(\overline{\epsilon}(\cC_{[0, \infty)}))
 \right) \circ 
\overline{\delta}_{\mathrm{L}}(r, D).
\end{align}
By applying the integration map $I_{r, D}$ in (\ref{integ}), 
using Lemma~\ref{standard} and the computation (\ref{RR}),
we obtain the formula
\begin{align}\notag
\PT_{r, D}(q, t)
&=\mathrm{L}_{r, D}(q, t) \\
\label{rat:form}&\cdot \exp\left(
\sum_{n\ge 0, \beta>0}
(-1)^{rn+D\beta-1}(rn-D\beta)N_{n, \beta}q^n t^{\beta}
\right). 
\end{align}
Therefore we obtain the desired result. 
\end{proof}
The result of Theorem~\ref{intro:rational}
is a corollary of the above result:
\begin{cor}\label{cor:rat}
Theorem~\ref{intro:rational} holds. 
\end{cor}
\begin{proof}
This is a consequence of the formula (\ref{rat:form})
together with Lemma~\ref{lem:Nrational} and Lemma~\ref{lem:Lpoly}. 
\end{proof}

\bibliographystyle{amsalpha}
\bibliography{math}

\providecommand{\bysame}{\leavevmode\hbox to3em{\hrulefill}\thinspace}
\providecommand{\MR}{\relax\ifhmode\unskip\space\fi MR }
\providecommand{\MRhref}[2]{%
  \href{http://www.ams.org/mathscinet-getitem?mr=#1}{#2}
}
\providecommand{\href}[2]{#2}
\begin{thebibliography}{BBBBJ15}

\bibitem[AHR]{AHR}
J.~Alper, J.~Hall, and D.~Rydh, \emph{A {L}una \'etale slice theorem for
  algeraic stacks}, preprint, arXiv:1504.06467.

\bibitem[Bay09]{Bay}
A.~Bayer, \emph{Polynomial {B}ridgeland stability conditions and the large
  volume limit}, Geom.~Topol.~ \textbf{13} (2009), 2389--2425.

\bibitem[BBBBJ15]{BBBJ}
O.~Ben-Bassat, C.~Brav, V.~Bussi, and D.~Joyce, \emph{A '{D}arboux {T}heorem'
  for shifted symplectic structures on derived {A}rtin stacks, with
  applications}, Geom.~Topol.~ \textbf{19} (2015), 1287--1359.

\bibitem[Beh09]{Beh}
K.~Behrend, \emph{Donaldson-{T}homas invariants via microlocal geometry},
  Ann.~of Math \textbf{170} (2009), 1307--1338.

\bibitem[BF08]{BBr}
K.~Behrend and B.~Fantechi, \emph{Symmetric obstruction theories and {H}ilbert
  schemes of points on threefolds}, Algebra Number Theory \textbf{2} (2008),
  313--345.

\bibitem[BG]{BG}
K.~Behrend and E.~Getzler, \emph{Chern-{S}imons functional}, in preparation.

\bibitem[BJM]{BDM}
V.~Bussi, D.~Joyce, and S.~Meinhardt, \emph{On motivic vanishing cycles of
  critical loci}, preprint, arXiv:1305.6428.

\bibitem[BMT14]{BMT}
A.~Bayer, E.~Macri, and Y.~Toda, \emph{Bridgeland stability conditions on
  3-folds {I}: {B}ogomolov-{G}ieseker type inequalities}, J.~Algebraic Geom.~
  \textbf{23} (2014), 117--163.

\bibitem[Bri07]{Brs1}
T.~Bridgeland, \emph{Stability conditions on triangulated categories}, Ann.~of
  Math \textbf{166} (2007), 317--345.

\bibitem[Bri11]{BrH}
\bysame, \emph{Hall algebras and curve-counting invariants},
  J.~Amer.~Math.~Soc.~ \textbf{24} (2011), 969--998.

\bibitem[Bri12]{BrI}
\bysame, \emph{An introduction to motivic {H}all algebras}, Adv.~Math.~
  \textbf{229} (2012), 102--138.

\bibitem[Bus]{Bussi}
V.~Bussi, \emph{Generalized {D}onaldson-{T}homas theory over fields ${K}\neq
  \mathbb{C}$}, preprint, arXiv:1403.2403.

\bibitem[CDP10]{CDP}
W.~Y. Chuang, D.~E. Diaconescu, and G.~Pan, \emph{Rank two {ADHM} invariants
  and wallcrossing}, Commun.~Number Theory Phys.~ \textbf{4} (2010), 417--461.

\bibitem[DM]{BM}
B.~Davison and S.~Meinhardt, \emph{Donaldson-{T}homas theory for ctegories of
  homological dimension one with potential}, preprint, arXiv:1512.08898.

\bibitem[HL97]{Hu}
D.~Huybrechts and M.~Lehn, \emph{Geometry of moduli spaces of sheaves}, Aspects
  in Mathematics, vol. E31, Vieweg, 1997.

\bibitem[HRS96]{HRS}
D.~Happel, I.~Reiten, and S.~O. Smal$\o$, \emph{Tilting in abelian categories
  and quasitilted algebras}, Mem.~Amer.~Math.~Soc, vol. 120, 1996.

\bibitem[HT10]{HT2}
D.~Huybrechts and R.~P. Thomas, \emph{Deformation-obstruction theory for
  complexes via {A}tiyah-{K}odaira-{S}pencer classes}, Math.~Ann.~ (2010),
  545--569.

\bibitem[Jiaa]{YJ2}
Y.~Jiang, \emph{On motivic {J}oyce-{S}ong formula for the {B}ehrend function
  identities}, preprint, arXiv:1601.00133.

\bibitem[Jiab]{YJ1}
\bysame, \emph{The {T}hom-{S}ebastiani theorem for the {E}uler characteristic
  of cyclic {$L$}-infinity algebras}, preprint, arXiv:1511.07912.

\bibitem[Joy07]{Joy3}
D.~Joyce, \emph{Configurations in abelian categories
  {I}\hspace{-.1em}{I}\hspace{-.1em}{I}. {S}tability conditions and
  identities}, Advances in Math \textbf{215} (2007), 153--219.

\bibitem[Joy08]{Joy4}
\bysame, \emph{Configurations in abelian categories {I}\hspace{-.1em}{V}.
  {I}nvariants and changing stability conditions}, Advances in Math
  \textbf{217} (2008), 125--204.

\bibitem[Joy15]{JoyceD}
\bysame, \emph{A classical model for derived critical loci}, J.~Differential
  Geom.~ \textbf{101} (2015), 289--367.

\bibitem[JS12]{JS}
D.~Joyce and Y.~Song, \emph{A theory of generalized {D}onaldson-{T}homas
  invariants}, Mem.~Amer.~Math.~Soc.~ \textbf{217} (2012).

\bibitem[KS]{K-S}
M.~Kontsevich and Y.~Soibelman, \emph{Stability structures, motivic
  {D}onaldson-{T}homas invariants and cluster transformations}, preprint,
  arXiv:0811.2435.

\bibitem[Lan09]{Langer2}
A.~Langer, \emph{Moduli spaces of sheaves and principal {$G$}-bundles},
  Proc.~Sympos.~Pure Math.~ \textbf{80} (2009), 273--308.

\bibitem[Le]{LeQu}
Q.~T. Le, \emph{Proofs of the integral identity conjecture over algebraically
  closed fields}, Duke math.~J.~ \textbf{164}, 157--194.

\bibitem[Li06]{Li}
J.~Li, \emph{Zero dimensional {D}onaldson-{T}homas invariants of threefolds},
  Geom.~Topol.~ \textbf{10} (2006), 2117--2171.

\bibitem[Lie06]{LIE}
M.~Lieblich, \emph{Moduli of complexes on a proper morphism}, J.~Algebraic
  Geom.~ \textbf{15} (2006), 175--206.

\bibitem[Lo12]{JLo}
J.~Lo, \emph{Polynomial {B}ridgeland {S}table {O}bjects and {R}eflexive
  {S}heaves}, Math.~Res.~Lett.~ \textbf{19} (2012), 873--885.

\bibitem[Lo13]{JLo2}
\bysame, \emph{Moduli of {PT}-{S}emistable {O}bjects {II}}, Trans
  Amer.~Math.~Soc.~ \textbf{365} (2013), 4539--4573.

\bibitem[LP09]{LP}
M.~Levine and R.~Pandharipande, \emph{Algebraic cobordism revisited},
  Invent.~Math.~ \textbf{176} (2009), 63--130.

\bibitem[MNOP06]{MNOP}
D.~Maulik, N.~Nekrasov, A.~Okounkov, and R.~Pandharipande,
  \emph{Gromov-{W}itten theory and {D}onaldson-{T}homas theory. {I}},
  Compositio.~Math \textbf{142} (2006), 1263--1285.

\bibitem[Nag]{Nhig}
K.~Nagao, \emph{On higher rank {D}onaldson-{T}homas invariants}, preprint,
  arXiv:1002.3608.

\bibitem[PP]{PP}
R.~Pandharipande and A.~Pixton, \emph{Gromov-{W}itten/{P}airs correspondence
  for the quintic 3-fold}, preprint, arXiv:1206.5490.

\bibitem[PT]{PiYT}
D.~Piyaratne and Y.~Toda, \emph{Moduli of {B}ridgeland semistable objects on
  3-folds and {D}onaldson-{T}homas invariants}, to appear in Crelle,
  arXiv:1504.01177.

\bibitem[PT09]{PT}
R.~Pandharipande and R.~P. Thomas, \emph{Curve counting via stable pairs in the
  derived category}, Invent.~Math.~ \textbf{178} (2009), 407--447.

\bibitem[PTVV13]{PTVV}
T.~Pantev, B.~To$\ddot{\textrm{e}}$n, M.~Vaquie, and G.~Vezzosi, \emph{Shifted
  symplectic structures}, Publ.~Math.~IHES \textbf{117} (2013), 271--328.

\bibitem[She11]{Shes}
A.~Sheshmani, \emph{Towards studying of the higher rank theory of stable
  pairs}, Thesis (Ph.~D.~) University of Illinois at Urbana-Champaign (2011).

\bibitem[ST11]{StTh}
J.~Stoppa and R.~P. Thomas, \emph{Hilbert schemes and stable pairs: {GIT} and
  derived category wall crossings}, Bull.~Soc.~Math.~France \textbf{139}
  (2011), 297--339.

\bibitem[Sto12]{Stop}
J.~Stoppa, \emph{D0-{D}6 states counting and {GW} invariants},
  Lett.~Math.~Phys.~ \textbf{102} (2012), 149--180.

\bibitem[Tho00]{Thom}
R.~P. Thomas, \emph{A holomorphic {C}asson invariant for {C}alabi-{Y}au 3-folds
  and bundles on ${K3}$-fibrations}, J.~Differential.~Geom \textbf{54} (2000),
  367--438.

\bibitem[Tod]{Tcurve4}
Y.~Toda, \emph{Stable pair invariants on {C}alabi-{Y}au 3-folds containing
  $\mathbb{P}^2$}, to appear in Geometry and Topology, arXiv:1410.3574.

\bibitem[Tod10a]{Tcurve1}
\bysame, \emph{Curve counting theories via stable objects~{I}: {DT/PT}
  correspondence}, J.~Amer.~Math.~Soc.~ \textbf{23} (2010), 1119--1157.

\bibitem[Tod10b]{Tolim2}
\bysame, \emph{Generating functions of stable pair invariants via
  wall-crossings in derived categories}, Adv.~Stud.~Pure Math.~ \textbf{59}
  (2010), 389--434, New developments in algebraic geometry, integrable systems
  and mirror symmetry (RIMS, Kyoto, 2008).

\bibitem[Tod10c]{Trk2}
\bysame, \emph{On a computation of rank two {D}onaldson-{T}homas invariants},
  Communications in Number Theory and Physics \textbf{4} (2010), 49--102.

\bibitem[Tod11]{Tcurve3}
\bysame, \emph{Curve counting invariants around the conifold point},
  J.~Differential Geom.~ \textbf{89} (2011), 133--184.

\bibitem[Tod12]{Tsurvey}
\bysame, \emph{Stability conditions and curve counting invariants on
  {C}alabi-{Y}au 3-folds}, Kyoto Journal of Mathematics \textbf{52} (2012),
  1--50.

\bibitem[Tod13a]{TodBG}
\bysame, \emph{Bogomolov-{G}ieseker type inequality and counting invariants},
  Journal of Topology \textbf{6} (2013), 217--250.

\bibitem[Tod13b]{Tcurve2}
\bysame, \emph{Curve counting theories via stable objects~{II}. {DT}/nc{DT}
  flop formula}, J.~Reine Angew.~Math.~ \textbf{675} (2013), 1--51.

\end{thebibliography}

Kavli Institute for the Physics and 
Mathematics of the Universe, University of Tokyo,
5-1-5 Kashiwanoha, Kashiwa, 277-8583, Japan.

\textit{E-mail address}: yukinobu.toda@ipmu.jp

\end{document}